\documentclass{amsart}

\usepackage{fullpage}

 \usepackage{amssymb,latexsym,amsmath,amsthm}

\usepackage[numbers,sort&compress]{natbib}

 \renewcommand{\div}{\mathop{\mathrm{div}}\nolimits}

 \renewcommand{\span}{\mathop{\mathrm{Span}}\nolimits}

\newtheorem*{thm*}{Theorem A}

\newtheorem{thm}{Theorem}[section]

\newtheorem{dfn}{Definition}[section]

\newtheorem{lemma}{Lemma}[section]
\newtheorem{prop}{Proposition}[section]

\numberwithin{equation}{section}

\begin{document}

\def\IR{{\mathbb{R}}}

\title[]{Stable solutions of symmetric systems involving hypoelliptic operators}

\author{Mostafa Fazly}

\address{Department of Mathematics,  University of Texas at San Antonio, San Antonio, TX 78249, USA}
\email{mostafa.fazly@utsa.edu}

\thanks{The author gratefully acknowledges University of Texas at San Antonio Start-up Grant.}

\maketitle

\begin{abstract}  
Let $X$ and $Y$ be two noncommuting vector fields in an open set $\Omega$  in a  manifold $\mathbb M$ equipped with a sub-Riemannian structure.    We examine stable solutions of the following symmetric system  
\begin{equation*}
\Delta_{XY}  u_i =  H_i(u_1,\cdots,u_m)  \ \ \text{in } \ \ \Omega \ \ \text{for}\ \ 1\le i \le m, 
  \end{equation*}  
when the operator $\Delta_{XY} $ is the H\"{o}rmander's operator given by $\Delta_{XY} (\cdot):= X(X\cdot) + Y(Y\cdot) $ and $H_i\in C^1(\mathbb R^m) $.  We prove the following identity for any $w \in C^2(\Omega)$
\begin{equation*}\label{} 
  |\nabla_{XY}X w|^2  +  |\nabla_{XY}Y w|^2  - | X| \nabla_{XY} w||^2 - | Y|\nabla_{XY} w||^2
 = 
\left\{ \begin{array}{ll}
                     |\nabla_{XY} w|^{2} \left[\mathcal A^2 +  \mathcal B^2 \right] & \hbox{in $\{|\nabla_{XY} w|>0\}\cap \Omega $,} \\
                       0 \ a.e.& \hbox{in $\{|\nabla_{XY} w|=0\}\cap \Omega$,}
                                                                       \end{array}
                    \right.   \end{equation*} 
where $\mathcal A$ is the intrinsic curvature of the level sets of $w$ and $\mathcal B $ is connected with the intrinsic normal and  the intrinsic tangent direction to the level sets and also with the Lie bracket $[X,Y]$.  We then apply this to establish a geometric Poincar\'{e} inequality for stable solutions of the above system for general vector fields $X$ and $Y$. This inequality enables us to analyze the level sets of stable solutions. In addition, we provide  certain reduction of dimensions results which can be regarded as counterparts of  the classical De Giorgi type results. This is remarkable since the  classical one-dimensional symmetry results do not hold for general vector fields.  Our approaches can be applied, but not limited, to the Grushin vector fields $X=(1,0)$ and $Y=(0,x)$ in $\mathbb R^2$ and the Heisenberg  vector fields $X=(1,0,-\frac{y}{2})$ and $Y=(0,1,\frac{x}{2})$ in $\mathbb R^3$ and their multidimensional extensions.  These specific vector fields generate nonelliptic operators which are hypoelliptic.

\end{abstract}

\noindent
{\it \footnotesize 2010 Mathematics Subject Classification}: {\scriptsize   35R03, 35A23, 53C17, 53C21, 35H10.}\\
{\it \footnotesize Keywords:   Hypoelliptic equations and systems, Heisenberg group,  Grushin plane, stable solutions, De Giorgi's conjecture}. {\scriptsize }

\section{Introduction} 
It is by now standard that any twice continuous differentiable solutions
of the Laplace's equation is an analytic function. An analogous property applies for any elliptic equations and systems
with analytic coefficients.  Inspired by this fact, the question of describing more general
differential operators $P(x,D)$ having the property that any solution
of the linear equation 
\begin{equation}
P(x,D)u = f \ \ \text{in} \ \ \Omega \subset\mathbb R^n \ \text{(or a manifold)},
\end{equation}
 is in $C^\infty(\Omega)$ when $f \in C^\infty(\Omega)$ was introduced  by Schwartz \cite{Sc} and H\"{o}rmander \cite{h1,h2,h3} and it has been studied extensively in the literature since then.  The literature in this context is too vast to give more than a few representative references \cite{rs,rock,m,Mil,gr,fo,fs,eg}.   Differential operators satisfying this property are called hypoelliptic. In addition, the heat equation operator
\begin{equation}\label{PxDt}
P(x,D)u=u_{t}- c \Delta_{x} u, 
\end{equation}
 when $c>0$ is hypoelliptic but not elliptic and the wave equation operator 
\begin{equation}\label{PxDtt}
P(x,D)u=u_{tt}- c^2 \Delta_{x} u ,
\end{equation}
 when $c\neq 0$ is not hypoelliptic. This implies that hypoellipticity is not necessarily ellipticity. However, every elliptic operator with $C^{\infty }$ coefficients is hypoelliptic.   Generally speaking, when the coefficients in the above operator $P$ are constant, a complete algebraic characterization of hypoelliptic operators are derived by  H\"{o}rmander in 1950's, see  \cite{h1}. However, the case of nonconstant coefficients is more challenging and some sufficient and necessary conditions are given by various authors, see \cite{h2,h3,m,Mil,rs} and references therein. 

In 1970,  Grushin \cite{gr} studied the following  class of differential operators
with polynomial coefficients which are not elliptic, but satisfy the H\"{o}rmander's conditions for hypoellipticity under certain assumptions; 
\begin{equation}\label{GrG}
P(x,D)u=\sum_{|\alpha|+|\beta| \le N, |\gamma|\le N \delta} a_{\alpha\beta\gamma} (x' )^\gamma D^\alpha_{x'} D^\beta_{x''}u, 
\end{equation}
for $x=(x',x'')\in \mathbb R^{n-k} \times \mathbb R^{k}$,  some constant  $a_{\alpha\beta\gamma}$ and  positive parameters $\delta,N$.  Here $\alpha$, $\beta$ and $\gamma$ are constant vectors. In particular, the second order operator 
\begin{equation}\label{Gr}
P(x,D)u=\Delta_{x'} u+ |x'|^{2s} \Delta_{x''} u , 
\end{equation}
is hypoelletic for any natural numbers $m,s$.  The operators given by (\ref{GrG}) and in particular (\ref{Gr}) are known as the Grushin operator.  The Grushin operator in two dimensions, and for $s=1$, is generated by the vector fields 
\begin{equation}\label{GrXY}
X = \partial_x \ \ \text{and} \ \ \ Y = x \partial_y.
\end{equation}
 In a more algebraic context, the Grushin plane
is defined by the vector fields $X = (1, 0)$ and $Y = (0, x)$ in two dimensions.  Note that $X$ and $Y$ do not commute and we denote the commutator by $[X, Y] = Z = \partial_y$.  Note also that $Z$ commutes with both $X$ and $Y$. The sub-Laplacian operator associated with the above vector fields is 
\begin{equation}\label{DeltaG}
\Delta_{\mathbb G} u:= X(X u) + Y(Yu) = \partial_{xx} + |x|^2 \partial_{yy}. 
\end{equation}
Another example for an operator with variable coefficients is the Heisenberg
group that plays an ubiquitous role in analysis and geometry.  The Heisenberg operator is generated by the vector fields 
\begin{equation}\label{HeisXY}
X = \partial_x - \frac{y}{2}\partial_z \ \ \text{and} \ \   Y =  \partial_y + \frac{x}{2} \partial_z. 
\end{equation}
Note that $X$ and $Y$ do not commute and we denote the commutator by $[X,Y]=Z=\partial_z$.  In other words,   the Heisenberg vector fields are given by $X=(1,0,-\frac{y}{2})$ and $Y=(0,1,\frac{x}{2})$ in $\mathbb R^3$.   Note that the sub-Laplacian operator associated with the above vector fields is 
\begin{equation}\label{Dh}
\Delta_{\mathbb H} u:= X(X u) + Y(Yu) = \Delta_{x,y} + \frac{1}{4} (|x|^2+|y|^2) \partial_{zz} - \partial_z (y\partial_x - x\partial_y) ,
\end{equation}
which coincides with the real part of the complex Kohn-Spencer sub-Laplacian operator.  It is known  that the analysis of the Grushin operator  is closely connected to that of the real part of Kohn-Spencer sub-Laplacian operator on the
Heisenberg group, at least in the case $s=1$. 

In the current article, we examine solutions of the following system of equations for a collection of smooth vector fields $X$ and $Y$  satisfying the
H\"{o}rmander bracket condition
\begin{equation}\label{mainL}
\Delta_{XY} u_i =  H_i(u_1,\cdots,u_m) \ \ \text{in} \ \ \Omega \ \ \ \text{for} \ \ 1\le i\le m, 
  \end{equation}  
when the operator $\Delta_{XY} $ is given by 
\begin{equation}\label{DeltaXY}
 \Delta_{XY} (\cdot) := X(X \cdot) + Y(Y \cdot ) , 
   \end{equation}  
and  $\Omega$ is a subset of $\mathbb R^n$ and $H_i\in C^1(\mathbb R^m) $ for $n,m\ge1 $.  The operator $\Delta_{XY}$ is known as the H\"{o}rmander's sum of squares.    The operator $\Delta_{XY}$ is in divergence form and it is given by 
\begin{equation}\label{LG}
\Delta_{XY} u_i = \div_{XY}( \nabla_{XY} u_i ) , 
  \end{equation}  
when  
\begin{equation}\label{mainG}
\nabla_{XY} u_i := (Xu_i) X + (Y u_i) Y. 
  \end{equation}  
  Here the divergence operator is of the following form 
  \begin{equation}
  \div_{XY}( \zeta) =  X \zeta_1 + Y \zeta_2, 
  \end{equation}
for $\zeta=\zeta_1 X+\zeta_2 Y$.   The notation $\langle , \rangle _{XY}$ stands for the standard scalar product, 
\begin{equation}
\langle \zeta, \theta \rangle_{XY}= \zeta_1\theta_1+ \zeta_2\theta_2, 
\end{equation}
for vector fields $\zeta=\zeta_1 X+\zeta_2 Y$ and $\theta=\theta_1 X+\theta_2 Y$. This in particular implies that 
\begin{equation*}
\langle \nabla_{XY} u_i ,\nabla_{XY} u_i  \rangle_{XY} = | Xu_i|^2 +| Yu_i|^2 \ \text{and} \ 
|  \nabla_{XY} u_i   | =\sqrt{\langle \nabla_{XY} u_i ,\nabla_{XY} u_i  \rangle_{XY} }. 
\end{equation*}
In order to illustrate the above formulation, consider the Euclidean space that is $\Omega=\mathbb R^2$ with $X=\partial_{x}$ and $Y=\partial_y$, then $\nabla_{XY}=(\partial_x,\partial_y)$ and $\Delta_{XY}=\partial_{xx}+\partial_{yy}$.  We now provide the notion of pointwise-stable  solutions. 

\begin{dfn} \label{stable}
A solution $u=(u_i)_{i=1}^m$ of (\ref{mainL}) is said to be pointwise-stable when there exists a sequence of functions $\phi=(\phi_i)_{i=1}^m $ where each $\phi_i\in C^3(\Omega)$ does not change sign such that the following linearized system holds
\begin{equation} \label{sta}
\Delta_{XY} \phi_i =   \sum_{j=1}^{m} \partial_j H_i(u) \phi_j  ,
 \end{equation} 
 for all $ i=1,\cdots,m$.  In addition, when $m\ge 2$ we assume that $\partial_j H_i(u) \phi_i \phi_j <0$ for all $1\le i\neq j\le m$. 
\end{dfn} 
Consider the case of scalar equations that is when $m=1$. Note that for the case of Euclidean space, monotonicity of a solution in any direction implies that the solution is pointwise-stable. For the case of general vector fields, this is more delicate and therefore more interesting. For the case of Grushin and Heienberg vector fields given in (\ref{GrXY}) and (\ref{HeisXY}), monotonicity only in the Lie bracket vector field $Z$,  that is $Zu$ does not change sign, 
 implies pointwise-stability.  For the case of system of equations, the monotonicity of solutions is introduced in \cite{fg,mf} in the context of Euclidean space. A solution $u=(u_k)_{k=1}^m$ of (\ref{mainL}) is said to be {\it $H$-monotone} if the following conditions hold
\begin{enumerate}
 \item[(i)] For every $1\le i \le m$, each $u_i$ is strictly monotone in the Lie bracket vector field $Z$ (i.e., $Z u_i\neq 0$).

\item[(ii)]  For $i,j=1,\cdots,m$,  we have 
  \begin{equation}
\hbox{$\partial_j H_i(u) Z u_i(x) Z u_j (x) < 0$  \ for all $x\in\Omega$.}
\end{equation}
\end{enumerate}
 We shall say that the system (\ref{mainL}) (or the nonlinearity $H$) is {\it orientable}, if there exists $\theta=(\theta_k)_{k=1}^m$ such that each $\theta_k$ is a nonzero function which does not change sign and  
 \begin{equation}
 \partial_j H_i(u) \theta_i(x)\theta_j(x) < 0 \ \ \text{for all} \ \ x\in\Omega,
  \end{equation} 
 for all $1\leq  i,j\leq m$. Note that the above condition on the system means that none of the mixed derivative $\partial_j H_i(u)$  changes sign.  It is straightforward to notice that $H$-monotonicity implies pointwise-stability. However, the reverse is not necessarily true.  We now provide the notion of symmetric systems introduced in \cite{mf,mf2} when $\Omega=\mathbb R^n$.  Symmetric systems play a fundamental role throughout this paper when we study system (\ref{mainL}) with a general nonlinearity $H(u)=(H_i(u))_{i=1}^m$.    
\begin{dfn}\label{symmetric} We call system (\ref{mainL}) symmetric if the matrix of gradient of all components of $H$ that is 
 \begin{equation} \label{H}
\mathcal {H}:=(\partial_i H_j(u))_{i,j=1}^{m} ,
 \end{equation}
  is symmetric. 
   \end{dfn}

\subsection{Scalar equations; case $m=1$} In the Euclidean space,  (\ref{mainL}) reads 
\begin{equation}\label{mainf}
\Delta u =  f(u) \ \ \ \text{in} \ \ \mathbb R^n. 
  \end{equation}   
In 1978,  De Giorgi \cite{deg} conjectured that bounded monotone solutions of Allen-Cahn equation, (\ref{mainf}) with $f(u)=u^3-u$, are one-dimensional solutions at least up to eight dimensions. In a more geometrical context, the statement implies that solutions must be such
that its level sets $\{u = \lambda\}$ are all hyperplanes.    There is an affirmative answer to this conjecture for almost all dimensions.  More precisely,  for two dimensions  Ghoussoub and Gui in \cite{gg1} and for three dimensions Ambrosio and Cabr\'{e} in \cite{ac} and with Alberti in \cite{aac} gave a proof to this conjecture not only for Allen-Cahn equation but also for a general nonlinearity $f$ that is locally Lipschitz.  We also refer interested readers to \cite{bcn}.  The conjecture remains open in dimensions $4\leq n \leq 8$. However, 
Ghoussoub and Gui  showed in \cite{gg2} that it is true for $n= 4$ and $n = 5$ for solutions that satisfy certain anti-symmetry conditions, and Savin \cite{sav} established its validity   for $4 \le n \le 8$ under the following additional natural hypothesis on the solutions,
 \begin{equation}\label{asymp}
\lim_{x_n\to\pm\infty } u({x}',x_n)\to \pm 1 \ \ \ \text{for} \ \ \ x' \in\mathbb R^{n-1}. 
\end{equation}
Unlike the above proofs in dimensions $n\leq 5$, the proof of Savin is nonvariational.  Note that there is an example by del Pino, Kowalczyk and Wei in \cite{dkw}  showing that eight dimensions is the critical dimension.  When the limits in (\ref{asymp}) are uniform, the De Giorgi's conjecture is known as the Gibbons' conjecture. This conjecture has been proved  for all dimensions independently by Barlow,  Bass and  Gui in \cite{bbg}, Berestycki,  Hamel and  Monneau in \cite{bhm} and Farina in \cite{f}. We also refer interested readers to Pacard and Wei in \cite{pw} for the stability conjecture that is when the monotonicity assumption is replaced with the stability notion in the De Giorgi's conjecture.

 Farina,  Sciunzi and Valdinoci in \cite{fsv} proved and applied a geometrical inequality, originally driven by Sternberg and Zumbrun in \cite{sz,sz1}, to provide various De Giorgi type results for elliptic equations in  two dimensions. This inequality has been of great interests in the literature and it has been established for various type equations on domains with diverse geometrical features.  In order to establish the inequality, the following identity is a crucial geometrical technique, established in \cite{sz,sz1}.  For any $w \in C^2(\Xi)$ when $\Xi\subset\mathbb R^n$ is any open set,  
 \begin{eqnarray}\label{identity} 
 \sum_{i=1}^{n} |\nabla \partial_i w|^2-|\nabla|\nabla w||^2
 = 
\left\{
                      \begin{array}{ll}
                       |\nabla w|^2 \left(\sum_{i=1}^{n-1} \mathcal{\kappa}_i^2\right) +|\nabla_{\bot }|\nabla w||^2 & \hbox{in $\{|\nabla w|>0\}\cap \Xi $,} \\
                       0 \ a.e.& \hbox{in $\{|\nabla w|=0\}\cap \Xi $,}
                                                                       \end{array}
                    \right.   \end{eqnarray} 
 where $ \mathcal{\kappa}_i$ are the principal curvatures of the level set of $w$ and $\nabla_\bot$ denotes the orthogonal projection of the gradient along this level set.    Farina,  Sire and Valdinoci in \cite{fsv1,fsv2} provided a similar inequality for stable solutions of elliptic equations on Riemannian manifolds and they applied it to establish Liouville type results on smooth, boundaryless Riemannian manifolds and the flatness of level sets of the solutions in  two dimensions.

Birindelli and Prajapat  in \cite{bp1} proved the Gibbons' conjecture in the context of Heisenberg group $\mathbb H$ for the equation 
\begin{equation}\label{mainH}
\Delta_{\mathbb H} u =  f(u),  
\end{equation}
for all directions orthogonal to the center of the Heisenberg group. However, the question of whether the conjecture holds in the direction of the center was left open in their work. Birindelli and Lanconelli in \cite{bl} proved the existence of a cylindrically symmetric solution to the above equation under the assumptions $Zu>0$, $|u|<1$  and 
\begin{equation}\label{asympt}
\lim_{z\to\pm\infty } u({x},y,z)\to \pm 1. 
\end{equation}
A solution $u(x,y,z)$ is called cylindrically symmetric if there exist a function $U$ such that $u(x,y, z) = U (r, z)$ when $r = |(x,y)|$ that implies 
\begin{equation}
\Delta_{\mathbb H} u (x,y,z) = \partial_{rr} U + \frac{1}{r} \partial_r{U} + 4 r^2 \partial_{zz} U. 
\end{equation}
As a direct consequence of the above result is that the De Giorgi conjecture is not true in the direction of the center of the group $\mathbb H$.  Ferrari and  Valdinoci in \cite{fv1} established a geometric inequality for stable solutions of (\ref{mainH}) and applied it to show that solutions have level sets with
vanishing mean curvature under certain assumptions.  In addition, Ferrari and  Valdinoci in \cite{fv2}  established  a geometric Poincar\'{e} type inequality for stable solutions of  the following equation on the Grushin plane 
\begin{equation}\label{mainG}
\Delta_{\mathbb G} u =  f(u).   
\end{equation}
Then, under certain assumptions on solutions, they applied the inequality together with some geometrical arguments to show that $u(x,y)$ depends only on the $x$-variable and to analyze the level sets of the solutions in two dimensions.  Inspired by the results provided in \cite{bl,bp1}, Birindelli and  Valdinoci in \cite{bv} proved the existence of a monotone solution that is $Zu>0$ for (\ref{mainG}) when 
 \begin{equation}\label{asympy}
\lim_{y\to\pm\infty } u(x,y)\to \pm 1 \ \ \text{for} \ \ \ x\in\mathbb R, 
\end{equation}
that is not one-dimensional.

\subsection{System of equations; case $m\ge 2$} In the Euclidean space,  (\ref{mainL}) is  
\begin{equation}\label{mainD}
\Delta u_i =  H_i(u_1,\cdots,u_m)  \ \ \text{in} \ \ \mathbb R^n.
  \end{equation} 
Ghoussoub and the author in \cite{fg}, and later in \cite{mf1,mf}, established De Giorgi type results for $H$-monotone and stable solutions of the above symmetric system in two and three dimensions.  Just like in the proof of the classical De Giorgi conjecture, the proof provided for the case of systems relies on a linear Liouville theorem and a geometrical Poincar\'{e} inequality. This system has been studied in the literature from various perspectives. Alama, Bronsard and Gui in \cite{abg}, proved that for certain nonlinearity $H$ and $m=2$ the system admits two-dimensional solutions that are not $H$-monotone. In addition,  in a series of articles in \cite{ali1,ali2,af} symmetry results and Liouville theorems, amongst other results, are proved.  The author considered the above system (\ref{mainD}) on Riemannian manifolds in \cite{mf2}  and established a geometrical inequality to study the level set of solutions and to prove Liouville theorems.

  In the current article,  we are interested in the analysis of level sets and in the reduction of dimensions results for stable solutions of (\ref{mainL}) with vector fields $X$ and $Y$.  In order to prove our main results, we establish a  geometric Poincar\'{e} type inequality for stable solutions following ideas and mathematical techniques provided in \cite{sz1,sz,fv1,fv2,fg,mf1,fsv} and references therein. In this regard, we establish a counterpart of the identity (\ref{identity}) for any $w \in C^2(\Omega)$ when $\Omega$ is any open set in a sub-Riemannian manifold $\mathbb M$ and for general vector fields $X$ and $Y$,
 \begin{eqnarray}\label{identity1} 
&& |\nabla_{XY}X w|^2  +  |\nabla_{XY}Y w|^2  - | X| \nabla_{XY} w||^2 - | Y|\nabla_{XY} w||^2
 \\&=& \nonumber
\left\{
                      \begin{array}{ll}
                     |\nabla_{XY} w|^{2} \left[\mathcal A^2 +  \mathcal B^2 \right] & \hbox{in $\{|\nabla_{XY} w|>0\}\cap \Omega $,} \\
                       0 \ a.e. & \hbox{in $\{|\nabla_{XY} w|=0\}\cap \Omega$,}
                                                                       \end{array}
                    \right.  
                     \end{eqnarray} 
where $\mathcal A$ is the intrinsic curvature of the level sets of each $u_i$ that is 
\begin{eqnarray}
\mathcal A &:=&    \div_{XY} \left( \eta  \right) , 
\\
\mathcal B &:=& |\nabla_{XY} w|^{-1}  \left[   -Zw + \langle H_{ess} w \tau, \eta \rangle \right]. 
\end{eqnarray}  
  Here, $\eta$ is the intrinsic normal and $\tau $ is the intrinsic tangent direction to the level sets of each $w$ on $\{ | \nabla_{XY} w | \neq 0 \}$ denoted by 
\begin{eqnarray}
\eta &:=&  |\nabla_{XY} w|^{-1} \left[   Xw X + Yw Y \right],
\\
\tau &:= & |\nabla_{XY} w|^{-1} \left[Yw X-Xw Y \right] .
\end{eqnarray}  
In addition,  the horizontal intrinsic Hessian matrix is 
\begin{equation}\label{Hess}
H_{ess} w:=
\begin{bmatrix}
    XXw & YXw  \\
     XYw & YYw 
     \end{bmatrix} . 
  \end{equation}  
Since two vector fields are non-commutative, $H_{ess} w$ is not symmetric.  We apply this  identity to prove that stable solutions of symmetric system (\ref{mainL}) satisfy the following geometric Poincar\'{e} type inequality 
\begin{eqnarray} \label{}
\label{ineq1} && \sum_{i=1}^m \int_{\Omega \cap \{ | \nabla_{XY} u_i | \neq 0 \}} |\nabla_{XY} u_i|^2 \left[ \mathcal A_i^2 + \mathcal B_i^2 \right] \zeta_i^2 
\\&& \label{ineq2}  -\sum_{i=1}^m \int_{\Omega } 2 \left[ ZYu_i Xu_i-ZXu_iYu_i   \right] \zeta_i^2 \\
 && \label{ineq3} +  \sum_{i\neq j=1}^m  \int_{\Omega }  \partial_j H_i(u) \left[ \langle  \nabla_{XY} u_i,\nabla_{XY} u_j\rangle_{XY}  \zeta_i^2 - |\nabla_{XY} u_i| |\nabla_{XY} u_j| \zeta_i \zeta_j  \right]
 \\&\le& \label{ineq4} \sum_{i=1}^m   \int_{\Omega }  |\nabla_{XY} u_i|^2  |\nabla_{XY} \zeta_i|^2 . 
\end{eqnarray} 
We then apply this inequality to establish the flatness of level sets for stable solutions of (\ref{mainL}) for various vector fields $X$ and $Y$. Unlike the classical De Giorgi type results, studied in \cite{fg,mf1,mf},  there is no bounded one-dimensional solution of the form $u_i(x,y)=f_i(ax+by)$ where $b\neq0$ and $u_i(x,y,z)=f_i(ax+by+cz)$ where $c\neq0$ for system (\ref{mainL}) with the Grushin and Heisenberg vector fields in two and three dimensions, respectively. This can be seen, as an example, by the following simple computations
\begin{equation}
(a^2+x^2b^2)f_i''(ax+by) = H_i(f_i(ax+by)), 
\end{equation}
where the left-hand side is not bounded and the right-hand side is bounded unless $f_i''\equiv0$ that implies a trivial case. In this regard, we prove that under certain assumptions, stable solutions of (\ref{mainL}) with the Grushin vector fields in two dimensions only depend on the $x$-variable. 

Let us end this section with pointing out that most of the results presented in this article can be generalized to multidimensional vector fields and/or fiber nonlinearities $H_i(x,u)$,  and for the sake of convenience of the readers we restrict ourselves to $X$ and $Y$ and nonlinearities $H_i(u)$.    Here is how this article is structured. In Section \ref{secpre},  we provide some basic information in regards to sub-Riemannian manifolds and hypoelliptic operators.   In Section \ref{secin}, we derive a stability inequality for stable solutions of system (\ref{mainL}). Then, we apply it to establish a geometric Poincar\'{e} inequality (\ref{ineq1})-(\ref{ineq4}). In addition, 
we provide a Hamiltonian identity for system (\ref{mainL}) with the Grushin operator. In Section \ref{secsym}, we apply the geometric Poincar\'{e} inequality to prove one-dimensional symmetry results and to analyze the level sets of stable solutions. 

\section{Preliminaries} \label{secpre}
A sub-Riemannian manifold is denoted by $(\mathbb M, \mathcal D, g)$  where $\mathbb M$ stands for a smooth
connected manifold, $\mathcal D$ is a smooth nonholonomic distribution  on $\mathbb M$ when $\mathcal D\subset T  \mathbb M$, and $g$ is a metric on $ \mathbb M$.     The geometry is defined on a manifold $ \mathbb M$, on which every trajectory evolves tangent to a distribution $\mathcal D$ of the tangent bundle $T \mathbb M$. Such trajectories are called horizontal curves. Riemannian geometry is the special case in which $\mathcal D = T \mathbb M$.  A metric, known as the sub-Riemannian metric, is defined as an inner product on the distribution.   We refer interested readers to \cite{cc,mont,br,gro} for more information.

A curve $\gamma:[0,1] \to \mathbb M$ is called a horizontal path with respect to $\mathcal D$ if it belongs to $W^{1,2}([0,1], \mathbb M)$ and satisfies 
\begin{equation}
\dot{\gamma}(t) \in \mathcal D(\gamma(t)) \ \ \text{for a.e.} \ \ t\in[0,1].
\end{equation}
The length of a path is defined by 
\begin{equation}
\text{length}_{g}(\gamma) := \int_0^1 \sqrt{ g_{\gamma(t)} (\dot{\gamma}(t), \dot{\gamma}(t))} dt. 
\end{equation}
The sub-Riemannian distance $d(x,y)$ between two points $x,y\in\mathbb M$ is the infimum over the lengths of the  horizontal paths between $x$ and $y$. This is also known as the Carnot-Charath\'{e}odory distance.

The Chow-Rashevsky theorem,  developed independently by Chow \cite{chow} and Rashevsky \cite{ras}, states that any two points of $\mathbb M$ can be joined by a horizontal path, if the Chow's conditions holds that is the vector fields $X$ and $Y$ and their iterated brackets span the tangent space $T_x \mathbb M$ at every point $\mathbb M$.  In other words, for any $x,y\in\mathbb M$ there exists a horizontal path $\gamma:[0,1]\to\mathbb M$ such that $\gamma(0)=x$ and $\gamma(1)=y$.    The Chow's condition is also known under the name of Lie algebra rank condition  since it states that the rank at every point $x$ of the Lie algebra generated by the $X$ and $Y$ is full. In the context of PDEs, it is known under the name of Hormander's condition that when it holds, the differential operator $\Delta_{XY}=X^2 + Y^2$ is hypoelliptic,  see Hormander's theorem \cite{h3}. Conversely, when $\mathbb M$ and the vector fields $X$ and $Y$ are analytic, the hypoellipticity of $\Delta_{XY}$ implies the Chow's condition. 
 More precisely, the H\"{o}rmander bracket condition
 states that 
\begin{equation}
T_m \mathbb M = \span(\{ \mathcal X(m); \ \mathcal X \in \mathcal L\}) \ \ \text{for all} \ m\in\mathbb M,
  \end{equation}  
when $\mathcal L$ is the Lie algebra of vector fields generated by the collection $\{ X,Y\}$.  The term comes from the analysis literature; it is named after H\"{o}rmander who obtained hypoellipticity results for linear operators associated with families of vector fields.  It is also referred as the bracket generating condition by Montgomery in \cite{mont} and as nonholonomic vector fields by Bella\"{i}che in \cite{br}.

 One of the simplest examples of the sub-Riemannian geometry is the  three-dimensional  Heisenberg geometry.   The Heisenberg algebra denoted by $\mathbb H$  is the three-dimensional Lie algebra with basis $\{X, Y, Z\}$ and with the only nonzero bracket between the basis elements being $[X, Y ] = Z$. Vector fields  $X, Y, Z$ are left-invariant vector fields on the corresponding simply connected Lie group $\mathbb H$ that is diffeomorphic to $\mathbb R^3$.  The underlying manifold of this Lie group is simply $\mathbb R^3$ with the non-commutative group
law 
\begin{equation}
(x,y,z)(x',y',z') = \left(x+x',y+y',z+z'+\frac{1}{2} (xy'-yx') \right) , 
\end{equation}
for $(x,y,z)$ and $(x',y',z') $ in $\mathbb R^3$. We define the distribution $\mathcal D$ on $\mathbb H$ to be the span of $X$ and $Y$  which we declare to be orthonormal. We refer interested readers to \cite{cdpt,fs,Mil,rock} and references therein.

 The next example for the sub-Riemannian, at least along some singular line, is the  Grusin plane.  The Grushin plane  is $\mathbb R^2$ endowed with the vector fields $X$ and $Y$ in (\ref{GrXY}). These vector fields span the tangent space everywhere, except along the line $x = 0$, where adding $Z=[X,Y]$ is needed.   The Grushin metric outside the critical line $x = 0$, the sub-Riemannian metric is in fact Riemannian, and is equal to
 \begin{equation}
g=dx^2+\frac{dy^2}{x^2}. 
\end{equation} 
 The metric can be extended continuously across the critical line $x = 0$ as a Carnot-Charath\'{e}odory metric, since $Z=[X,Y]=\partial_y \neq 0$. Any path has finite length, provided its tangent is parallel to the $x$-axis when crossing the $y$-axis. We refer interested readers to \cite{beck,mm,gr,fv2,rock} and references therein. 
 
 We end this section with a technical lemma regarding the commutation of vector fields $X$ and $Y$ and the operator $\Delta_{XY} $ given by system (\ref{mainL}).  

\begin{lemma}\label{lemtech}  
  Let $u=(u_i)_{i=1}^m$ be a solution of system (\ref{mainL}).  Then,  
\begin{equation}
\Delta_{XY} Xu_i + 2Z Yu_i   =\sum_{j=1}^m \partial_j H_i(u) X u_j \ \  \text{and} \ \ 
\Delta_{XY}  Yu_i - 2Z Xu_i   = \sum_{j=1}^m \partial_j H_i(u) Y u_j , 
\end{equation}  
where $Z=[X,Y]$. 
\end{lemma}  
\begin{proof} 
From (\ref{mainL}),  we get  
\begin{equation}\label{XDH}
X\Delta_{XY} u_i =  \sum_{j=1}^m \partial_j H_i(u) X u_j  \ \ \text{and} \ \  Y\Delta_{XY} u_i =  \sum_{j=1}^m \partial_j H_i(u) Y u_j    . 
\end{equation}
Straightforward computations show that 
\begin{equation}\label{XD}
X \Delta_{XY} u_i  =  \Delta_{XY} X u_i + 2ZY u_i \ \ \text{and} \ \ 
Y \Delta_{XY} u_i  =  \Delta_{XY} Y u_i - 2ZX u_i. 
\end{equation}
Combining (\ref{XDH}) and (\ref{XD}) completes the proof. 

\end{proof}

 Although the setting of this article is mostly the hypoelliptic calculus on Heisenberg groups and Grushin planes, we expect that most of our results can be extended
to more general settings such as the hypoelliptic calculus on Carnot-Carath\'{e}odory
manifolds. 


\section{Inequalities for stable solutions and a Hamiltonian identity} \label{secin}

In this section, we provide a stability inequality for pointwise-stable solutions of symmetric system (\ref{mainL}). Then, we apply this stability inequality to establish the geometric Poincar\'{e} inequality (\ref{ineq1})-(\ref{ineq4}) for general vector fields $X$ and $Y$ that is not limited to the Heisenberg vector fields (\ref{HeisXY}), Grushin vector fields (\ref{GrXY}) and Martinet vector fields 
\begin{equation}
X=\partial_x \ \ \text{and} \ \ \ Y=\partial_y +\frac{x^2}{2} \partial_z.
\end{equation}
Throughout this article, we call $u=(u_i)_{i=1}^m$ a stable solution of (\ref{mainL}) if it is a solution of (\ref{mainL}) and it satisfies the stability inequality (\ref{stability}). 

\begin{lemma}\label{lemstable}  
  Let $u$ be a pointwise-stable solution of symmetric system (\ref{mainL}).  Then 
\begin{equation} \label{stability}
-\sum_{i,j=1}^{m} \int_{\Omega} \partial_i H_j(u) \zeta_i \zeta_j  \le \sum_{i=1}^{m} \int_{\Omega}  \langle \nabla_{XY} \zeta_i, \nabla_{XY} \zeta_i\rangle_{XY}  , 
\end{equation} 
for any $\zeta=(\zeta_i)_{i=1}^m$ where $ \zeta_i \in C_c^1( \Omega )$ for $1\le i\le m$. 
\end{lemma}  
\begin{proof}   Since $u=(u_i)_{i=1}^m$ is a stable solution, there exists a  sequence $\phi=(\phi_i)_{i=1}^m$  such that for all $i=1,\cdots, m$  
\begin{equation}\label{linearzeta}
\Delta_{XY}  \phi_i =  \sum_{j=1}^{m} \partial_j H_i(u) \phi_j  . 
\end{equation}
Consider test function  $\zeta=(\zeta_i)_{i=1}^m$ where $ \zeta_i \in C_c^1(\Omega)$ for $1\le i\le m$.  Multiplying  both sides of (\ref{linearzeta})  with $\frac{\zeta_i^2}{\phi_i}$ and integrating, for each $i=1,\cdots,m$ we get 
\begin{equation}\label{ineq}
-\sum_{j=1}^{m}  \int_{\Omega}  \partial_j H_i(u) \phi_j \frac{\zeta_i^2}{\phi_i}    = -\int_{\Omega}   \frac{\Delta_{XY}  \phi_i}{\phi_i} {\zeta_i^2}  . 
  \end{equation}
Applying the Cauchy-Schwarz inequality,  we conclude
 \begin{equation}
- \int_{\Omega}   \frac{\Delta_{XY}  \phi_i}{\phi_i} {\zeta_i^2}   \le \int_{\Omega}  \langle \nabla_{XY} \zeta_i, \nabla_{XY} \zeta_i \rangle_{XY}   . 
   \end{equation}
In above the following inequality, that is inspired by the Cauchy-Schwarz inequality, is applied 
\begin{equation}
 2  \langle \frac{\nabla_{XY} \phi_i}{\phi_i } , \nabla_{XY} \zeta_i \rangle _{XY}   \zeta_i - \langle \frac{\nabla_{XY} \phi_i}{\phi_i},  \frac{\nabla_{XY} \phi_i}{\phi_i} \rangle _{XY}  \zeta_i^2 \le \langle \nabla_{XY}\zeta_i,\nabla_{XY}\zeta_i   \rangle _{XY}  .
 \end{equation}
 Note also that here we have applied the divergence theorem for the Laplace-Beltrami operator. For the left-hand side of (\ref{ineq}), straightforward calculations show that  
\begin{eqnarray}
  \sum_{i,j=1}^{m}  \int_{\Omega}  \partial_j H_i(u) \phi_j \frac{\zeta_i^2}{\phi_i} &=&   \sum_{i<j}^{m}  \int_{\Omega}  \partial_j H_i(u) \phi_j \frac{\zeta_i^2}{\phi_i}
 + \sum_{i>j}^{n}  \int_{\Omega}  \partial_j H_i(u) \phi_j \frac{\zeta_i^2}{\phi_i} 
  +  \sum_{i=1}^{m}\int_{\Omega} \partial_i H_i(u) {\zeta_i^2}
   \\&=& \nonumber \sum_{i<j}^{m}   \int_{\Omega} \partial_j H_i(u) \phi_j \frac{\zeta_i^2}{\phi_i}  + \sum_{i<j}^{m}  \int_{\Omega}  \partial_i H_j(u) \phi_i \frac{\zeta_j^2}{\phi_j}  + \sum_{i=1}^{m} \int_{\Omega} \partial_i H_i(u) {\zeta_i^2}
  \\&=&\nonumber  \sum_{i<j}^{m}   \int_{\Omega}  \partial_i H_j(u) \phi_i\phi_j \left(   \frac{\zeta^2_i}{\phi^2_i}  +   \frac{\zeta^2_j}{\phi^2_j}  \right)  
 +  \sum_{i=1}^{m} \int_{\Omega}  \partial_i H_i(u) {\zeta_i^2}  
  \\&\le & \nonumber  2 \sum_{i<j}^{m}  \int_{\Omega}  {\partial_j H_i(u) } \zeta_i \zeta_j  +  \sum_{i=1}^{m} \int_{\Omega}  \partial_i H_i(u) {\zeta_i^2} 
  \\&=& \nonumber \sum_{i,j=1}^{m}   \int_{\Omega} {\partial_j H_i(u) } \zeta_i\zeta_j  ,
  \end{eqnarray}
where we have used $\partial_i H_j(u) \phi_i\phi_j<0$ and $   \frac{\zeta^2_i}{\phi^2_i}  +   \frac{\zeta^2_j}{\phi^2_j}  \ge 2\frac{\zeta_i\zeta_j}{\phi_i\phi_j}$.  This ends the proof. 

\end{proof}   

We now apply the stability inequality to prove a Poincar\'{e} type inequality. 

\begin{prop}\label{propine}  
  Let $u=(u_i)_{i=1}^m$ be a stable solution of symmetric system (\ref{mainL}) for general vector fields $X$ and $Y$.  Then,  
\begin{eqnarray}
&& \label{XYT1} 
 \sum_{i=1}^m \int_{\overline \Omega} \left[ |\nabla_{XY}X u_i|^2  +  |\nabla_{XY}Y u_i|^2  - | X| \nabla_{XY} u_i||^2 -| Y|\nabla_{XY} u_i||^2\right] \zeta_i^2
\\&& \label{XYT2}  -\sum_{i=1}^m   \int_{\Omega} 2 \left[ ZYu_i Xu_i-ZXu_iYu_i   \right] \zeta_i^2 
\\&& \label{XYT3} +  \sum_{i\neq j=1}^m  \int_{\Omega}  \partial_j H_i(u) \left[ \langle  \nabla_{XY} u_i, \nabla_{XY} u_j \rangle_{XY}   \zeta_i^2 - | \nabla_{XY} u_i| | \nabla_{XY} u_i| \zeta_i \zeta_j  \right]
\\&\le& \label{XYT4} \sum_{i=1}^m   \int_{\Omega}   |\nabla_{XY} u_i|^2  | \nabla_{XY} \zeta_i|^2, 
\end{eqnarray}  
where $\overline \Omega=\Omega \cap \{ | \nabla_{XY} u_i | \neq 0 \}$ and each $\zeta_i$ is a test function $\zeta_i\in C_c^2(\Omega)$. 
\end{prop} 
 
\begin{proof} Multiplying equations given in Lemma \ref{lemtech} by  $X u_i \zeta_i^2$ and $ Y u_i \zeta_i^2$,  we get  
\begin{eqnarray}\label{XXD}
 \sum_{j=1}^m \partial_j H_i(u) X u_j  X u_i \zeta_i^2 &=&  XX (X u_i) X u_i \zeta_i^2  +YY (X u_i)   X u_i \zeta_i^2 + 2Z Yu_i  X u_i \zeta_i^2  , \\
 \sum_{j=1}^m \partial_j H_i(u) Yu_j  Y u_i \zeta_i^2  &=& XX (Y u_i)  Y u_i \zeta_i^2+YY (Y u_i) Y u_i \zeta_i^2 
 - 2Z Xu_i  Y u_i \zeta_i^2    . 
\end{eqnarray}
Integrating by parts yields  
\begin{eqnarray}\label{HijXX}
 \sum_{j=1}^m \int_{\Omega}    \partial_j H_i(u) X u_j  X u_i \zeta_i^2&=& -\int_{\Omega} (XX u_i)\left[ (X X u_i) \zeta_i^2  +Xu_i X \zeta^2_i \right] 
\\&& \nonumber -\int_{\Omega} (YX u_i) \left[(Y X u_i) \zeta_i^2  +Xu_i Y \zeta^2_i\right] 
 + 2   \int_{\Omega}Z Yu_i  X u_i \zeta_i^2   ,
\end{eqnarray}
and similarly, 
\begin{eqnarray}\label{HijYY}
 \sum_{j=1}^m \int_{\Omega}   \partial_j H_i(u) Y u_j  Y u_i \zeta_i^2&=& -\int_{\Omega} (XY u_i) \left[(X Y u_i) \zeta_i^2  +Yu_i X \zeta^2_i \right]
\\&& \nonumber -\int_{\Omega} (YY u_i) \left[(Y Y u_i) \zeta_i^2  +Yu_i Y \zeta^2_i \right]
 - 2   \int_{\Omega}Z Xu_i  Y u_i \zeta_i^2   . 
\end{eqnarray}
Combining the above equations (\ref{HijXX}) and (\ref{HijYY}),  we get 
\begin{eqnarray}
 - \int_{\Omega}   \partial_i H_i(u) \left[ (X u_i)^2+  (Y u_i)^2\right] \zeta_i^2  
&=& \label{Summ1} \sum_{j\neq i}^m \int_{\Omega}   \partial_j H_i(u) \left(Y u_j  Y u_i + X u_j  X u_i  \right)\zeta_i^2
\\&&  \label{Summ2} +\int_{\Omega} \left[ (XX u_i)^2 +  (X Y u_i)^2 +(YX u_i)^2+(Y Y u_i)^2    \right] \zeta_i^2
\\&&  \label{Summ3} -2 \int_{\Omega} \left[ Z Yu_i  X u_i  -Z Xu_i  Y u_i \right]\zeta_i^2  
\\&& \label{Summ4} +\int_{\Omega} (XX u_i)(X u_i) (X \zeta_i^2) + (YX u_i)(X u_i) (Y \zeta_i^2) 
\\&& \label{Summ5}  +\int_{\Omega} (XY u_i)(Y u_i) (X \zeta_i^2) +(YY u_i)(Y u_i) (Y \zeta_i^2).  
\end{eqnarray}
We now apply the stability inequality (\ref{stability}) for the test functions $\zeta_i$ replaced by $|\nabla_{XY} u_i|\zeta_i$
\begin{equation} \label{}
-\sum_{i,j=1}^{m} \int_{\Omega} \partial_i H_j(u) |\nabla_{XY} u_i| |\nabla_{XY} u_j| \zeta_i \zeta_j  \le \sum_{i=1}^{m} \int_{\Omega}  \langle  |\nabla_{XY} u_i| \zeta_i, |\nabla_{XY} u_i| \zeta_i\rangle_{XY}  .
\end{equation} 
Rearranging the terms we obtain 
\begin{eqnarray} \label{sumOmegaXY}
&&-\sum_{i=1}^{m} \int_{\Omega} \partial_i H_i(u) |\nabla_{XY} u_i|^2 \zeta_i^2  -\sum_{i\neq j=1}^{m} \int_{\Omega} \partial_i H_j(u) |\nabla_{XY} u_i| |\nabla_{XY} u_j| \zeta_i \zeta_j 
\\& \le&  \nonumber  \sum_{i=1}^{m} \int_{\Omega}    |X(|\nabla_{XY} u_i| \zeta_i)|^2 +   |Y(|\nabla_{XY} u_i| \zeta_i)|^2    . 
\end{eqnarray} 
Straightforward computations show that 
\begin{eqnarray}\label{XnablaXX}
 |X(|\nabla_{XY} u_i| \zeta_i)|^2 &= & |\nabla_{XY} u_i|^{-1} \left[ (Xu_i)(XXu_i)+(Yu_i)(XYu_i)\right]^2 \zeta_i^2 
\\&& \nonumber   +2 \left[ (Xu_i)(XXu_i)+(Yu_i)(XYu_i)\right]  \zeta_i X\zeta_i 
+ |\nabla_{XY} u_i|^{2} (X \zeta_i)^2.  
\end{eqnarray}
Similarly, 
\begin{eqnarray}\label{YnablaXX}
 |Y(|\nabla_{XY} u_i| \zeta_i)|^2 &= & |\nabla_{XY} u_i|^{-1} \left[ (Xu_i)(YXu_i)+(Yu_i)(YYu_i)\right]^2 \zeta_i^2 
\\&& \nonumber  +2 \left[ (Xu_i)(YXu_i)+(Yu_i)(YYu_i)\right]  \zeta_i Y\zeta_i
+  |\nabla_{XY} u_i|^{2} (Y \zeta_i)^2.  
\end{eqnarray}
Combining (\ref{XnablaXX}), (\ref{YnablaXX}) and (\ref{sumOmegaXY}) we conclude
\begin{eqnarray} 
-\sum_{i=1}^{m} \int_{\Omega} \partial_i H_i(u) |\nabla_{XY} u_i|^2 \zeta_i^2 
&\le& \label{Sum1} \sum_{i\neq j=1}^{m} \int_{\Omega} \partial_i H_j(u) |\nabla_{XY} u_i| |\nabla_{XY} u_j| \zeta_i \zeta_j 
\\& & \label{Sum2}  +  \sum_{i=1}^{m} \int_{\Omega}  \left[|X|\nabla_{XY} u_i||^2  +   |Y|\nabla_{XY} u_i||^2 \right]  \zeta_i^2
\\&& + \sum_{i=1}^{m} \int_{\Omega} |\nabla_{XY} u_i|^{2} \left[  (X \zeta_i)^2 +   (Y \zeta_i)^2 \right]
\\&&\label{Sum3}  +  \sum_{i=1}^{m} \int_{\Omega}  \left[ (Xu_i)(YXu_i)+(Yu_i)(YYu_i)\right]   Y\zeta_i^2
\\&&\label{Sum4} +   \sum_{i=1}^{m} \int_{\Omega} \left[ (Xu_i)(XXu_i)+(Yu_i)(XYu_i)\right]  X\zeta^2_i . 
\end{eqnarray} 
We now combine the equality (\ref{Summ1})-(\ref{Summ5}) and the inequality (\ref{Sum1})-(\ref{Sum4}). Notice that each term in (\ref{Sum3})-(\ref{Sum4}) and in (\ref{Summ4})-(\ref{Summ5}) are identical. This implies that 
 \begin{eqnarray} \label{Sumf0}
&&\sum_{i=1}^m  \int_{\Omega} \left[ (XX u_i)^2 +  (X Y u_i)^2 +(YX u_i)^2+(Y Y u_i)^2 \right]\zeta_i^2
\\&&-  \sum_{i=1}^m  \int_{\Omega}  \left[ | X| \nabla_{XY} u_i||^2 + | Y|\nabla_{XY} u_i||^2   \right] \zeta_i^2
\\&&-2 \sum_{i=1}^m  \int_{\Omega} \left[ Z Yu_i  X u_i  -Z Xu_i  Y u_i \right]\zeta_i^2 
\\&&+\sum_{j\neq i}^m \int_{\Omega}   \partial_j H_i(u) \left[ \left(Y u_j  Y u_i + X u_j  X u_i \right)\zeta_i^2 -   |\nabla_{XY} u_i| |\nabla_{XY} u_j| \zeta_i \zeta_j \right] 
\\&\le& 
\sum_{i=1}^m   \int_{\Omega}   |\nabla_{XY} u_i|^2  | \nabla_{XY} \zeta_i|^2. 
\end{eqnarray} 
Note that $Y u_j  Y u_i + X u_j  X u_i = \langle  \nabla_{XY} u_i,\nabla_{XY} u_j\rangle_{XY} $.  This completes the proof. 

\end{proof}   

In what follows,  we show that all the terms in (\ref{XYT1}), i.e. 
\begin{equation}\label{XYidentity}
 \left[ |\nabla_{XY}X u_i|^2  +  |\nabla_{XY}Y u_i|^2  - | X| \nabla_{XY} u_i||^2 -| Y|\nabla_{XY} u_i||^2\right] , 
\end{equation} 
has a geometric representation.  In the Euclidean sense, Sternberg and Zumbrun in \cite{sz,sz1} derived the  geometric identity (\ref{identity}) between the tangential gradients and curvatures in the Euclidean sense. Inspired by this identity, we provide an identity for all terms in (\ref{XYidentity}). This enables us to conclude that (\ref{XYT1}) has a  positive fixed sign and this proves the geometric Poincar\'{e} type inequality (\ref{ineq1})-(\ref{ineq4}).   
     
\begin{thm}\label{thmine}  
  Let $u=(u_i)_{i=1}^m$ be a stable solution of symmetric system (\ref{mainL}).  Then, the geometric Poincar\'{e} type inequality (\ref{ineq1})-(\ref{ineq4}) holds. 
\end{thm}  

\begin{proof}  Proposition \ref{propine} implies that the inequality (\ref{XYT1})-(\ref{XYT4}) holds. With loss of generality we suppose that $|\nabla_{XY} u_i |\neq 0 $. Note that on $\{|\nabla_{XY} u_i|= 0\}$ we have $ Z Yu_i  X u_i  -Z Xu_i  Y u_i =0$ and  $\nabla_{XY} Xu_i=\nabla_{XY} Yu_i=0$ a.e. for each index $i$.  It is straightforward to compute  
\begin{eqnarray}
| X| \nabla_{XY} u_i||^2  &=& \frac{1}{(Xu_i)^2 + (Yu_i)^2} \left[ (Xu_i) (XXu_i) +(Yu_i) (XYu_i) \right]^2,\\
| Y| \nabla_{XY} u_i||^2  &=& \frac{1}{(Xu_i)^2 + (Yu_i)^2} \left[ (Xu_i) (YXu_i) +(Yu_i) (YYu_i) \right]^2. 
\end{eqnarray}
Applying the above equalities,  doing the expansion in the right-hand side and collecting like terms,   one can get  
\begin{eqnarray}
\mathcal C_i&:=& (XX u_i)^2 +  (X Y u_i)^2 +(YX u_i)^2+(Y Y u_i)^2 - | X| \nabla_{XY} u_i||^2 -| Y|\nabla_{XY} u_i||^2
\\&=& \nonumber \frac{1}{(Xu_i)^2 + (Yu_i)^2} \left[ (Yu_i)(XXu_i) - (Xu_i)(XYu_i) \right]^2 
\\&&\nonumber + \frac{1}{(Xu_i)^2 + (Yu_i)^2} \left[(Xu_i)(YYu_i)-(Yu_i)(YXu_i) \right]^2 
\\&=& \nonumber \left[ \frac{Yu_i}{ |\nabla_{XY} u_i|  }(XXu_i) -  \frac{Xu_i}{ |\nabla_{XY} u_i|  }(XYu_i) \right]^2  
\\&&\nonumber +
\left[  \frac{Yu_i}{ |\nabla_{XY} u_i|  }(YXu_i) -\frac{Xu_i}{ |\nabla_{XY} u_i|  }(YYu_i)  \right]^2 .
\end{eqnarray}
Notice that the latter is nonenegative and we show that it is in fact related to the  intrinsic curvature of the level sets, just like (\ref{identity}).  As it was introduced earlier, the intrinsic curvature of the level sets of each $u_i$ is 
\begin{equation}
\mathcal A_i =    \div_{XY} \left( \eta_i  \right) = X\eta_i^{(1)} + Y \eta_i^{(2)} ,  
\end{equation}  
where $ \eta_i  =\eta_i^{(1)}  X+ \eta_i^{(2)} Y$ that is the intrinsic normal to the level sets of each $u_i$ on $\{ | \nabla_{XY} u_i | \neq 0 \}$  and denoted by 
\begin{equation}\label{etaiXui}
\eta_i=   \frac{Xu_i}{  |\nabla_{XY} u_i|}    X + \frac{Yu_i}{ |\nabla_{XY} u_i|} Y. 
\end{equation} 
It is straightforward to compute  the curvature and show that
\begin{equation}\label{Ai}
 |\nabla_{XY} u_i|^{3} \mathcal A_i =(XXu_i) (Yu_i)^2 -  (XYu_i) (Xu_i)  (Yu_i)  -(YXu_i) (Xu_i)  (Yu_i) + (YYu_i) (Xu_i)^2   .
\end{equation}  
On the other hand, $\mathcal B_i$ is given by the following formula 
 \begin{equation}\label{BBi}
\mathcal B_i := |\nabla_{XY} u_i|^{-1}  \left[   -Zu_i + \langle H_{ess} u_i \tau_i, \eta_i \rangle \right] , 
\end{equation}  
where  $\eta_i$ is the intrinsic normal 
and $\tau_i $  is the intrinsic tangent direction to the level sets of each $u_i$ on $\{ | \nabla_{XY} u_i | \neq 0 \}$ that is (\ref{etaiXui}) and 
\begin{equation}
\tau_i=   \frac{Yu_i}{  |\nabla_{XY} u_i|}    X - \frac{Xu_i}{ |\nabla_{XY} u_i|} Y. 
\end{equation} 
Again, doing some  straightforward  computations one can get  
\begin{equation}\label{Bi}
 |\nabla_{XY} u_i|^{3} \mathcal B_i =(XXu_i) (Xu_i) (Yu_i) -  (XYu_i) (Xu_i)^2 
   + (YXu_i) (Yu_i)^2 - (YYu_i) (Xu_i) (Yu_i)   .
\end{equation}  
We can rewrite the equation for $\mathcal C_i$ as
\begin{eqnarray}
|\nabla_{XY} u_i|^{4} \mathcal C_i&=&
\left[ (Yu_i)(XXu_i) -  (Xu_i) (XYu_i) \right]^2[(Xu_i)^2 +(Yu_i)^2 ] 
\\&& \nonumber +
  \left[  (Yu_i) (YXu_i) -(Xu_i)(YYu_i)  \right]^2 [(Xu_i)^2 +(Yu_i)^2 ].
  \end{eqnarray}
 Elementary computations show that for all $a,b,c,d,\epsilon,\delta\in\mathbb R$ we have 
   \begin{equation}\label{}
\left[ a\epsilon -b \delta\right]^2[\epsilon^2+\delta^2 ]+  \left[ c\epsilon -d \delta \right]^2[\epsilon^2+\delta^2 ] 
= \left[a\epsilon^2 - (b+c)\epsilon\delta +d\delta^2\right]^2 + \left[c\epsilon^2+(a-d)\epsilon\delta-b\delta^2\right]^2. 
  \end{equation}
We now set $a=XXu_i$, $b=XYu_i$, $c=YXu_i$, $d=YYu_i$, $\epsilon=Xu_i$ and $\delta=Yu_i$ in the above identity to conclude that 
\begin{eqnarray*}\label{}
|\nabla_{XY} u_i|^{4} \mathcal C_i
&=& \left[ (Yu_i)(XXu_i) -  (Xu_i) (XYu_i) \right]^2[(Xu_i)^2 +(Yu_i)^2 ] 
\\&& +
  \left[  (Yu_i) (YXu_i) -(Xu_i)(YYu_i)  \right]^2 [(Xu_i)^2 +(Yu_i)^2 ]
  \\&=& \left[ (XXu_i) (Yu_i)^2 -  (XYu_i) (Xu_i)  (Yu_i) -(YXu_i) (Xu_i)  (Yu_i) + (YYu_i) (Xu_i)^2   \right]^2
  \\& &+ \left[ (XXu_i) (Xu_i) (Yu_i) -  (XYu_i) (Xu_i)^2 + (YXu_i) (Yu_i)^2 - (YYu_i) (Xu_i) (Yu_i)  \right]^2
  \\&=& |\nabla_{XY} u_i|^{6} \mathcal A^2_i +  |\nabla_{XY} u_i|^{6} \mathcal B^2_i,
  \end{eqnarray*}
where we have used (\ref{Ai}) and (\ref{Bi}). Therefore, 
\begin{equation}\label{nablaC}
\mathcal C_i =  |\nabla_{XY} u_i|^{2} \left[\mathcal A^2_i +  \mathcal B^2_i \right]. 
  \end{equation}
The above arguments imply that the following identity holds 
 \begin{eqnarray}\label{identity2} 
&& (XX u_i)^2 +  (X Y u_i)^2 +(YX u_i)^2+(Y Y u_i)^2 - | X| \nabla_{XY} u_i||^2 -| Y|\nabla_{XY} u_i||^2
 \\&=& \nonumber
\left\{
                      \begin{array}{ll}
                     |\nabla_{XY} u_i|^{2} \left[\mathcal A^2_i +  \mathcal B^2_i \right] & \hbox{in $\{|\nabla_{XY} u_i|>0\cap \Omega \}$,} \\
                       0 & \hbox{$a.e.$ in $\{|\nabla_{XY} u_i|=0\cap \Omega\}$,}
                                                                       \end{array}
                    \right.   \end{eqnarray} 
where $\mathcal A_i$ is the intrinsic curvature of the level sets of each $u_i$ and  $
\mathcal B_i $ is denoted by (\ref{BBi}).  This and Proposition \ref{propine} complete the proof. 

\end{proof}   

Consider the following symmetric system of ordinary differential equations  
\begin{equation}\label{ODE}
u_i''(x)=  H_i(u_1(x),\cdots,u_m(x)) \ \ \text{for} \ \ x\in\mathbb R, 
\end{equation}  
that is a particular case of system (\ref{mainL}). It is straightforward to see that the  following Hamiltonian identity holds for  solutions of (\ref{ODE}) 
\begin{equation}\label{ODE11}
\frac{1}{2} \sum_{i=1}^m \left[ u_i' (x) \right]^2 -\tilde H (u_1(x),\cdots,u_m(x)) \equiv C\ \ \text{for} \ \ x\in\mathbb R,
\end{equation}  
 where $\partial_i \tilde H = H_i$ and $C$ is a constant.   
 For the rest of this section, in light of the above identity, we establish a Hamiltonian identity for solutions of  (\ref{mainL}) with  slightly more general operator for $x=(x',x'')\in\mathbb R^{n-k}\times \mathbb R^k$, $n\ge k$ and $k\ge 1$, 
\begin{equation}\label{Opf}
\Delta_{XY} (\cdot) = X_{x'}(X_{x'}(\cdot)) +  f^2 (x') Y_{x''}(Y_{x''}(\cdot)), 
\end{equation}
where  the vector fields $ X_{x'}$ and $ Y_{x''}$ depend only on $x'$ and $x''$, respectively, and $f\in C(\mathbb R^{n-k},\mathbb R)$. 
This in particular covers the  Grushin operator $\Delta_{\mathbb G}$ in two dimensions, given by (\ref{DeltaG}), when $X_x=\partial_{x}$ and $Y_y= \partial_{y}$ and  $f(x)=x$.  Note that for the case of $f\equiv 1$, the operator (\ref{Opf}) is the classical Laplacian operator. For the case of Laplacian operator, similar identities are provided by Gui in \cite{g,g1} and for the case of quasilinear operators by the author in \cite{mf1}.

 \begin{thm}\label{thmHam}  
 Suppose that $x=(x',x'')\in\mathbb R^{n-1}\times \mathbb R$ for $n\ge 2$.  Let $u=(u_i)_{i=1}^m$  be a solution of (\ref{mainL}) in $\mathbb R^n$ for the  operator  given by (\ref{Opf}) when $f=(f_i(x'))_{i=1}^m$ for $f_i\in C(\mathbb R^{n-1},\mathbb R)$. Then there exists a constant $C$ such that  the following Hamiltonian identity holds for every $x'' \in \mathbb R$
  \begin{equation}\label{hamil}
\int_{\mathbb R^{n-1}} \left( \sum_{i=1}^{m} \frac{1}{2}  \left[ |X_{x'} u_i(x)|^2 - f_i^2(x') |Y_{x''}u_i(x)|^2  \right] + \tilde H(u(x))  \right) d x' \equiv C,
\end{equation}
 where $\tilde H$ is defined such that $\partial_i\tilde H(u)=H_i(u)$ and the above integral is finite for at least one value of $x''$ and in addition the integral in (\ref{IRI0})
below tends to zero as $R$ goes to infinity along a sequence.
\end{thm}  
\begin{proof} Suppose that $x=(x',x'')\in\mathbb R^{n-1}\times \mathbb R$ and assume that $B_R$ is a ball of radius $R$ in $\mathbb R^{n-1}$.  Set $I:\mathbb R\to \mathbb R$ as  
\begin{equation}\label{IRx}
I_R(x''):=\int_{B_R}\left( \sum_{i=1}^{m} \frac{1}{2}  \left[ |X_{x'} u_i(x)|^2 -f_i^2(x') |Y_{x''}u_i(x)|^2  \right] + \tilde H(u(x))  \right) d x'. 
\end{equation}
Differentiating $I$ with respect to $x''$,  we obtain 
\begin{eqnarray} \label{}
\partial_{x''} I_R(x'') &=& \sum_{i=1}^{m}  \int_{B_R}   \left[  X_{x'} u_i(x)   X_{x'} \partial_{x''} u_i(x) - f_i^2(x') Y_{x''}u_i(x) Y_{x''} \partial_{x''}u_i(x)  \right]  d x'
\\&& \nonumber \label{IRx2} +  \sum_{j=1}^{m}  \int_{B_R}  H_j(u(x)) \partial_{x''}u_j(x)   d x'. 
\end{eqnarray}
Multiplying both sides of (\ref{mainL}) with $ \partial_{x''}u_j(x)$,  for each index $j$,  we get
\begin{equation}\label{Deltax}
H_j(u(x))  \partial_{x''}u_j(x) =  X_{x'}(X_{x'}(u_j))  \partial_{x''}u_j(x)  +  f_j^2(x')  Y_{x''}(Y_{x''}(u_j))\partial_{x''}u_j(x).
\end{equation}  
Combining the above equalities and applying integration by parts, we conclude 
\begin{eqnarray}\label{IRxpp}
\partial_{x''} I_R(x'') &=& \sum_{i=1}^{m}  \int_{B_R}\left[ X_{x'} u_i(x)   X_{x'} \partial_{x''} u_i(x)  +   X_{x'}(X_{x'}(u_i))  \partial_{x''}u_i(x) \right] dx'
\\&=& \nonumber \sum_{i=1}^{m}  \int_{\partial B_R}\left[  \partial_{\nu(x')} u_i(x)   \partial_{x''}u_i(x) \right] d S(x').
\end{eqnarray}
Integrating the above equation with respect to $x''$, we get 
\begin{equation}\label{IRI0}
 I_R(x'')  -  I_R(0)  = \int_{0}^{x''}  \int_{\partial B_R}\left[  \partial_{\nu(x')} u_i(x)   \partial_{x''}u_i(x) \right] d S(x').
\end{equation}
Sending $R$ to infinity and using the assumptions, the right-hand side of the above equality approaches zero. This completes the proof. 
 
\end{proof}

\section{Flatness  of level sets and reduction of dimensions}\label{secsym}

We start this section with the flatness of level sets for stable solutions of symmetric system (\ref{mainL})  for the Grushin and the Heisenberg vector fields under certain assumptions. This is a consequence of the geometric Poincar\'{e} inequality and   the ideas and mathematical techniques applied in the proofs  can be extended to larger classes of vector fields. At the end of this section, we provide some geometric justifications in regards to the assumption (\ref{TYX}). Note that  when  the vector fields $X$ and $Y$ are commutative (\ref{TYX}) is clearly satisfied. 
 
\begin{thm} \label{thmABG}
 Suppose that  $u=(u_i)_{i=1}^m\in C^2(\mathbb G,\mathbb R^m)$ is a stable solution of symmetric system (\ref{mainL}) in two dimensions for the Grushin vector fields satisfying (\ref{GrXY}) and  
 \begin{equation}\label{TYX}
 Z Yu_i Xu_i- ZXu_iYu_i \le 0. 
 \end{equation}
 In addition, assume that   
\begin{equation}\label{BRGrow}
\int_{B_R} x^2 |\nabla_{\mathbb G} u_i|^2 dx dy\le C R^4 \ \ \text{for all}\ \ R>1\ \  \text{and} \ \ \ 1\le i\le m, 
\end{equation} 
for any $R>1$ when   $B_R=\{ (x,y) \in \mathbb G, x^4 +y^2<R^4  \}$.  Then, the intrinsic curvature of the level sets of $u_i$  vanishes identically that is $\mathcal A_i\equiv 0$ and  in addition $\mathcal B_i\equiv 0$ on $\{|\nabla_{\mathbb G} u_i| \neq 0\}$. 

  \end{thm}
 
\begin{proof}  Consider the Grushin norm 
\begin{equation}
||(x,y)||_{\mathbb G}=\left( x^4 +y^2 \right)^{\frac{1}{4}} , 
\end{equation}
 and set the following test function for $R>1$, 
 \begin{equation}\label{chiRnorm}
 \chi_R  (x,y):=\left\{
                      \begin{array}{ll}
                        \frac{1}{2}, & \hbox{if $||(x,y)||_{\mathbb G}\le\sqrt{R}$,} \\
                      \frac{ \ln R - \ln ||(x,y)||_{\mathbb G}}{{\ln R}}, & \hbox{if $\sqrt{R}< ||(x,y)||_{\mathbb G}< R$,} \\
                       0, & \hbox{if $||(x,y)||_{\mathbb G}\ge R$.}
                                                                       \end{array}
                    \right.
                      \end{equation} 
Since the system (\ref{mainL}) is {\it orientable}, there exist nonzero functions $\theta_k$, $k=1,\cdots,m$, which do not change sign such that for all  $ i,j\in\{1,\cdots,m\}$ and $i<j$ 
 \begin{equation}\label{sign}
  \partial_j H_{i}(u)\theta_i\theta_j < 0. 
  \end{equation} 
Consider  $\zeta_k:=sgn (\theta_k) \chi_R$ for $1\le k\le m$, where again $sgn (x)$ is the sign function. The geometric  Poincar\'{e} inequality (\ref{ineq1})-(\ref{ineq3}) and the assumption  (\ref{TYX}) imply 
\begin{eqnarray} \label{}
\label{ineq1chi} && \sum_{i=1}^m \int_{B_R \cap \{ | \nabla_{XY} u_i | \neq 0 \}} |\nabla_{XY} u_i|^2 \left[ \mathcal A_i^2 + \mathcal B_i^2 \right] \chi_R^2  \\
 && \label{ineq2chi} +  \sum_{i\neq j=1}^m  \int_{B_{R}}  |\partial_j H_i(u)|  \left[    |\nabla_{XY} u_i| |\nabla_{XY} u_j|    - \theta_i \theta_j  \langle  \nabla_{XY} u_i,\nabla_{XY} u_j\rangle_{XY}   \right]   \chi_R^2
 \\&\le& \label{ineq3chi} \sum_{i=1}^m  \int_{B_R\setminus B_{\sqrt {R} } }  |\nabla_{XY} u_i|^2  |\nabla_{XY}  \chi_R|^2 . 
\end{eqnarray} 
Note that (\ref{ineq1chi}) and (\ref{ineq2chi}) are nonnegative. We now compute (\ref{ineq3chi}). Note that for $(x,y)\in B_R\setminus B_{\sqrt {R} }$ we have 
\begin{equation}
|X\chi_R| = \frac{|x^3|}{ ||(x,y)||^4_{\mathbb G} \ln R} \ \ \text{and} \ \ |Y\chi_R| = \frac{|xy|}{ 2||(x,y)||^4_{\mathbb G} \ln R} ,
\end{equation}
and therefore 
\begin{equation}
|X\chi_R|^2+|Y\chi_R|^2 \le  \frac{x^2}{ ||(x,y)||^4_{\mathbb G} |\ln R|^2 } . 
\end{equation}
From this and the assumption (\ref{BRGrow}), we get 
\begin{equation}\label{nablaXYchi}
\int_{B_R\setminus B_{\sqrt {R} } }  |\nabla_{XY} u_i|^2  |\nabla_{XY}  \chi_R|^2  \le \frac{1}{|\ln R|^2} \int_{B_R\setminus B_{\sqrt {R} } }  \frac{x^2  |\nabla_{XY} u_i|^2 }{ ||(x,y)||^4_{\mathbb G}  } \le   \frac{C}{\ln R}. 
\end{equation}
Sending $R$ to infinity completes the proof. 

\end{proof}   

Note that if we assume that $|\nabla_{\mathbb G} u_i| \in L^{\infty}(\mathbb R^2)$. Then,   
\begin{equation}\label{}
\int_{B_R} x^2 |\nabla_{\mathbb G} u_i|^2 dx dy \le \int_{B_R} x^2 dxdy = C R^5, 
\end{equation} 
where $C$ is a constant that is independent from $R$.  Therefore, the global boundedness of $|\nabla_{\mathbb G} u_i|$ does not imply 
the growth-decay assumption (\ref{BRGrow}). In the following theorem, we provide sufficient conditions that under which the growth-decay assumption (\ref{BRGrow}) holds.   Notice that (\ref{energybound}) implies  (\ref{BRGrow}) whenever $\tilde H(\cdot) \ge \tilde H(a)$.  

\begin{thm} \label{enopthm}
 Suppose that  $u=(u_i)_{i=1}^m$ is a bounded solution of (\ref{mainL}) in two dimensions for Grushin vector fields satisfying (\ref{GrXY}).  In addition, assume that $| \nabla_{\mathbb G}  u_i|\in L^\infty(\mathbb R^2)$ and each $Z u_i$ does not change sign  for all $i=1,..,m$ and 
\begin{equation}
\lim_{y\to \infty} u_i( x,y)=a_i, \ \ \ \forall x\in\mathbb{R}, 
\end{equation}
for some constants $a_i$.  Then 
\begin{equation}\label{energybound}
I_R(u):=\int_{B_{R}}  \left[  \sum_{i=1}^m \frac{1}{2} |\nabla_{\mathbb G} u_i|^2 +  \tilde H(u) - \tilde H(a) \right] d x dy \le C R^{2},
\end{equation}
where  $a=(a_i)_{i=1}^{m}$ and $C$ is  independent from $R$.
 \end{thm}
 \begin{proof}  Set the sequence of shift functions $u^t(x,y)=(u_i^t(x,y))_{i=1}^m$  where  $u_i^t(x,y):=u_i(x,y+t)$ for $t\in\mathbb{R}$. Note that $u^t=(u_i^t)_{i=1}^m$ satisfies 
\begin{equation}
\label{maintt}
\Delta_{\mathbb G} u^t_i  =   H_i(u_1^t, \cdots, u_m^t). 
  \end{equation}
The fact that $u_i^t$ is convergent to $a_i$ pointwise, it is straightforward to  see that 
\begin{equation}\label{Ht}
 \int_{B_{R}} (\tilde H(u^t) -\tilde H(a)) dx dy \to 0 \ \ \text{when} \ \ {t\to\infty}.  
 \end{equation}
Now, multiply both sides  of (\ref{maintt}) with $u_i^t-a_i$ and integrate by parts to get
\begin{equation} \label{}
\hfill -\int_{B_{R}}  |\nabla_{\mathbb G} u_i^t|^2 dxdy + \int_{\partial B_{R}}  \partial_\nu u_i^t (u_i^t-a_i)  dS(x,y)=-\int_{B_{R}}   H_i(u^t) (u_i^t-a_i)  dxdy  .
  \end{equation}
Sending $t\to \infty$ in the above,  we get 
\begin{equation}
\int_{B_{R}} |\nabla_{\mathbb G} u_i^t|^2 dxdy \to 0. 
 \end{equation}
From this and (\ref{Ht}),  we conclude  
\begin{equation}\label{energydecayt}
\lim_{t\to\infty} I_R(u^t)=0.
\end{equation}
Differentiating the functional $I_R(u^t)$ with respect to $t$ and using the main equation (\ref{mainL}), we get 
\begin{eqnarray}\label{}
\partial_t I_R(u^t)&=& \sum_{i=1}^m \int_{B_{R}} \left[ \langle \nabla_{\mathbb G} u_i^t , \nabla_{\mathbb G} (Z u_i^t) \rangle_{\mathbb G} +  H_i(u^t) (Z u_i^t) \right] dxdy
\\&=&  \nonumber  \sum_{i=1}^m \int_{B_{R}} \left[ \langle \nabla_{\mathbb G} u_i^t , \nabla_{\mathbb G} (Z u_i^t) \rangle_{\mathbb G} +  (\Delta_{\mathbb G} u^t_i) (Z u_i^t) \right]dxdy
\\&=&  \nonumber  \sum_{i=1}^m \int_{B_{R}}  \left [ \partial_x u^t_i \partial_x Z u^t_i + x^2 \partial_y u^t_i \partial_y Z u^t_i  + 
 \partial_{xx} u^t_i  Z u^t_i + \partial_{yy} u^t_i  Z u^t_i \right] dxdy
 \\&=&  \nonumber R \sum_{i=1}^m \int_{B_{1}}  \left[\partial_{\bar x} u^t_i \partial_{\bar x} Z u^t_i + {\bar x}^2 \partial_{\bar y} u^t_i \partial_{\bar y} Z u^t_i  + 
 \partial_{\bar x \bar x} u^t_i  Z u^t_i + \partial_{\bar y \bar y} u^t_i  Z u^t_i \right] d\bar xd\bar y . 
    \end{eqnarray}
We now apply the standard divergence theorem in the Euclidean sense to conclude 
\begin{eqnarray}\label{ptIR}
\partial_t I_R(u^t)&=& R \sum_{i=1}^m \int_{B_{1}} \div \left[ Z u^t_i \left( \partial_{\bar x}  u^t_i ,  {\bar x}^2 \partial_{\bar y} u^t_i  \right)   \right] d\bar xd\bar y 
\\&=& R^2 \sum_{i=1}^m \int_{\partial B_{1}} Z u^t_i \left( \partial_{x}  u^t_i ,  R{\bar x}^2 \partial_{ y} u^t_i   \right) \cdot \nu(\bar x, \bar y) dS(\bar x,\bar y) , 
    \end{eqnarray}
 where $\nu$ is the standard Euclidean outward normal
of $\partial B_1$.  From the assumptions we have $|\nabla_{\mathbb G} u_i|\in L^{\infty}(\mathbb R^2) $.  Therefore, there exist a constant $M$ that is independent from $R$ such that $|\nabla_{\mathbb G} u_i| \le M$ and 
 \begin{eqnarray}
 |\left( \partial_{x}  u^t_i ,  R{ \bar x}^2 \partial_{ y} u^t_i   \right) \cdot \nu(\bar x,\bar y) | 
 \le  M \ \ \text{for} \ \ (\bar x, \bar y)\in \partial B_1.
 \end{eqnarray}
 On the other hand,  assumptions yield each $Z u^t_i$ does not change sign  for all indices. Therefore,  $Z u_k^t>0>Z u_j^t$ for $k\in  K$ and $j\in  J$ where $J$ and $K$ are disjoint sets such that $J\cup K=\{1,\cdots,m\}$. Therefore, 
  \begin{eqnarray}\label{der-energy-bdMt}
\partial_t I_R(u^t) \ge M R^2  \int_{\partial B_{1}}  \left(\Sigma _{j\in  J} Z u_j^t - \Sigma_{k\in  K} Z u_i^t \right)d S(\bar x,\bar y).
   \end{eqnarray}
Applying the above estimate, we can establish the following upper bound for $I_R(u)$  
 \begin{eqnarray}\label{}
 I_R(u)&=& I_R(u^t)- \int_{0}^{t} \partial_t I_R(u^s) ds 
 \\ &\le&\nonumber I_R(u^t) +M R^2 \int_{0}^{t}  \int_{\partial B_{1}} \left(  \Sigma_{k\in  K} Z u_k^s-  \Sigma_{j\in  J} Z u_j^s \right) dS(\bar x,\bar y) ds
 \\ &=&\nonumber I_R(u^t) +M R^2 \int_{\partial B_{1}}  \left(  \Sigma_{k\in  K} (u_k^t-u_k)+ \Sigma_{j\in  J}(u_j-u_j^t)  \right) dS(\bar x,\bar y).
 \end{eqnarray}
From the monotonicity assumptions, we have  $u_k<u_k^t$ and $u_j^t<u_j$ for all $k\in  K$, $j\in  J$ and $t\in \mathbb{R^+}$. Therefore, 
\begin{equation}\label{energytt}
  I_R(u) \le I_R(u^t) +C R^2 \ \ \ \text{for all} \ \ t\in \mathbb{R^+} , 
\end{equation}  
for $\max\{4M ||u_i||_{L^\infty(\mathbb R^2)},  |\partial B_{1}| \}\le C$ where $C$ is independent from $R$.   Sending  $t\to\infty$ and using  (\ref{energydecayt}), we conclude  that  
 \begin{equation}
 I_R(u) \le  C R^{2}. 
  \end{equation}
This completes the proof.  
 
 \end{proof}   

As it was mentioned earlier,  the results of Theorem \ref{thmABG} hold for general vector fields $X$ and $Y$ whenever the left-hand side of  (\ref{nablaXYchi}) decays to zero,   that is 
\begin{equation}\label{NablaXYchiXY}
\int_{B_R\setminus B_{\sqrt {R} } }  |\nabla_{XY} u_i|^2  |\nabla_{XY}  \chi_R|^2 \to 0 \ \ \text{for} \ \ \ R\to \infty , 
\end{equation}
 for the test function $\chi_R $ that satisfies (\ref{chiRnorm}). We are now ready to prove the flatness of level sets for the  Heisenberg vector fields.  

\begin{thm} \label{thmABH}
 Suppose that  $u=(u_i)_{i=1}^m\in C^2(\mathbb H,\mathbb R^m)$ is a stable solution of symmetric system  (\ref{mainL}) in three dimensions for the Heisenberg vector fields (\ref{HeisXY}), 
 \begin{equation}\label{}
 ZYu_i Xu_i-ZXu_iYu_i \le 0. 
 \end{equation}
 In addition, assume that  
\begin{equation}\label{BRGrowH}
\int_{B_R} (x^2+y^2) |\nabla_{\mathbb H} u_i|^2 dxdydz\le C R^4 \ \text{for all}\  R>1\   \text{and} \  1\le i\le m, 
\end{equation} 
for any $R>1$ when   $B_R=\{ (x,y,z) \in \mathbb H, |x^2+y^2|^2 +z^2<R^4  \}$.   Then, the intrinsic curvature of the level sets of $u_i$  vanishes identically that is $\mathcal A_i\equiv 0$ and  in addition $\mathcal B_i\equiv 0$ on $\{|\nabla_{\mathbb H} u_i| \neq 0\}$. 

  \end{thm}
\begin{proof}  Due to similarity in the arguments, we only show that (\ref{NablaXYchiXY}) holds.  Consider the Heisenberg norm 
\begin{equation}
||(x,y,z)||_{\mathbb H}=\left( |x^2+y^2|^2 +z^2 \right)^{\frac{1}{4}},
\end{equation} 
and set the following test function for $R>1$, 
 \begin{equation}
 \chi_R  (x,y,z):=\left\{
                      \begin{array}{ll}
                        \frac{1}{2}, & \hbox{if $||(x,y,z)||_{\mathbb H}\le\sqrt{R}$,} \\
                      \frac{ \log R - \ln ||(x,y,z)||_{\mathbb H}}{{\ln R}}, & \hbox{if $\sqrt{R}< ||(x,y,z)||_{\mathbb H}< R$,} \\
                       0, & \hbox{if $||(x,y,z)||_{\mathbb H}\ge R$.}
                                                                       \end{array}
                    \right.
                      \end{equation} 
Note that for $(x,y,z)\in B_R\setminus B_{\sqrt {R} }$ we have 
\begin{equation}
|X\chi_R| =2  \frac{ |(x^2+y^2)x + yz|}{ ||(x,y,z)||^4_{\mathbb H} \ln R} \ \ \text{and} \ \ |Y\chi_R| = 2 \frac{  |(x^2+y^2)y - yz|}{ ||(x,y,z)||^4_{\mathbb H} \ln R}   , 
\end{equation}
and therefore 
\begin{equation}
|X\chi_R|^2+|Y\chi_R|^2  =   4\frac{x^2+y^2}{ ||(x,y,z)||^4_{\mathbb H} |\ln R|^2 } . 
\end{equation}
From this and the assumption (\ref{BRGrowH}), we get 
\begin{equation}
\int_{B_R\setminus B_{\sqrt {R} } }  |\nabla_{XY} u_i|^2  |\nabla_{XY}  \chi_R|^2  \le \frac{1}{|\ln R|^2} \int_{B_R\setminus B_{\sqrt {R} } }  \frac{(x^2+y^2)  |\nabla_{\mathbb H} u_i|^2 }{ ||(x,y,z)||^4_{\mathbb H}  } \le   \frac{C}{\ln R}. 
\end{equation}
Sending $R$ to infinity completes the proof. 

\end{proof}   
 
Just like in the case of Grushin vector fields,    the global boundedness of $|\nabla_{\mathbb H} u_i|$ does not imply the growth-decay assumption (\ref{BRGrowH}) for the case of Heisenberg vector fields. In addition, straightforward computations show that a counterpart of the energy arguments  as given in Theorem \ref{enopthm}  does not guarantee the assumption (\ref{BRGrowH}) either.

In what follows,  we discuss consequences of $\mathcal A_i\equiv 0$ and  $\mathcal B_i\equiv 0$ on $\{|\nabla_{XY} u_i| \neq 0\}$ and derive certain equations.    If the intrinsic curvature of the level sets of $u_i$  vanishes identically that is $\mathcal A_i\equiv 0$ on $\{|\nabla_{XY} u_i| \neq 0\}$, then 
\begin{eqnarray}\label{Ai1}
0= |\nabla_{XY} u_i|^{3} \mathcal A_i &=&(XXu_i) (Yu_i)^2 -  (XYu_i) (Xu_i)  (Yu_i)
  \\&&\nonumber -(YXu_i) (Xu_i)  (Yu_i) + (YYu_i) (Xu_i)^2. 
\end{eqnarray}  
In addition,  $\mathcal B_i\equiv 0$ on $\{|\nabla_{XY} u_i| \neq 0\}$ implies 
\begin{eqnarray}\label{Bi1}
 0=|\nabla_{XY} u_i|^{3} \mathcal B_i &=&(XXu_i) (Xu_i) (Yu_i) -  (XYu_i) (Xu_i)^2 
 \\&& \nonumber + (YXu_i) (Yu_i)^2 - (YYu_i) (Xu_i) (Yu_i)   .
\end{eqnarray}  
Combining (\ref{Ai1}) and (\ref{Bi1}), via substitution of $(Xu_i)(Yu_i)$ from (\ref{Ai1}) into  (\ref{Bi1}) whenever $XXu_i\neq YYu_i$, we get 
\begin{eqnarray}\label{XuiYui1}
&&(Xu_i)^2 \left[  (YYu_i)^2+ (XYu_i)^2 + (XYu_i) (YXu_i) -(XXu_i) (YYu_i) \right]
\\&=& \label{XuiYui2} (Yu_i)^2 \left[  (XXu_i)^2 +(YXu_i)^2 + (XYu_i) (YXu_i) -(XXu_i) (YYu_i) \right]. 
\end{eqnarray}  
On the other hand, straightforward computations show that
\begin{equation}\label{Hessutau}
\langle H_{ess} u_i\tau_i,\tau_i \rangle_{XY} =  |\nabla_{XY} u_i|^{-2}  |\nabla_{XY} u_i|^{3}   \mathcal A_i =  |\nabla_{XY} u_i|    \mathcal A_i =0,
\end{equation}  
where we have used $0=\mathcal A_i $. We now apply $0=\mathcal B_i $ and (\ref{Hessutau}) to conclude that
\begin{eqnarray}\label{}
H_{ess} u_i\tau_i = \langle H_{ess} u_i\tau_i,\eta_i \rangle_{XY} \eta_i + \langle H_{ess} u_i\tau_i,\tau_i \rangle_{XY} \tau_i= Zu_i \eta_i .
\end{eqnarray}  
 Since $Zu_i=XYu_i-YXu_i$, we get the following 
\begin{eqnarray}\label{}
(XXu_i)(Yu_i) - (XYu_i)(Xu_i)&=&0, 
\\
(YXu_i)(Yu_i) - (YYu_i)(Xu_i)&=&0. 
\end{eqnarray}  
This implies that 
\begin{equation}\label{Hessui}
\det(H_{ess}u_i)=0. 
\end{equation}
This and (\ref{XuiYui1})-(\ref{XuiYui2}) imply 
\begin{equation}\label{XuiYui3}
(Xu_i)^2 \left[  (YYu_i)^2+ (XYu_i)^2\right]
=  (Yu_i)^2 \left[  (XXu_i)^2 +(YXu_i)^2 \right], 
\end{equation}  
that is equivalent to 
 \begin{equation}\label{XuiYui4}
|Xu_i| | \nabla_{XY} Yu_i| 
=  |Yu_i| | \nabla_{XY} Xu_i| . 
\end{equation}  
Therefore, $\mathcal A_i\equiv 0$ and  $\mathcal B_i\equiv 0$ imply  equations (\ref{Hessui}) and (\ref{XuiYui4}) which are Monge-Amp\'{e}re type equations for the level sets.  

In the next theorem, we establish reduction of dimensions for stable solutions of (\ref{mainL}) in two dimensions for the Grushin vector fields.  This can be regarded as De Giorgi type results,  one-dimensional symmetry, for the Grushin operator. We refer interested readers to \cite{fg} for the classical De Giorgi type results in the Euclidean sense for the system of equations.   
 
\begin{thm} \label{thmred}
 Suppose that  $u=(u_i)_{i=1}^m\in C^2(\mathbb G,\mathbb R^m)$ is a  stable solution of (\ref{mainL}) in two dimensions for the Grushin vector fields satisfying (\ref{GrXY}) and 
 \begin{equation}
 ZYu_i Xu_i-ZXu_iYu_i \le 0. 
 \end{equation}
 In addition, assume that $|\nabla_{\mathbb G} u_i| \neq 0$ and  
\begin{equation}\label{BRNablauR4}
\int_{B_R} x^2 |\nabla_{\mathbb G} u_i|^2 dx dy\le C R^4  \ \ \text{for} \ \ R>1.
\end{equation} 
 Then, each $u_i$ depends only on the variable $x$.  
  \end{thm}
 \begin{proof}  
 As it was shown in the proof of Theorem \ref{thmABG}, the intrinsic curvature of the level sets of $u_i$  vanishes identically that is $\mathcal A_i\equiv 0$ and $|\nabla_{\mathbb G} u_i|\neq 0$. We provide the proof in several steps.

 \noindent {\bf Step 1}. Suppose that $\Gamma_i$ is a connected component of the level set $\{(x,y)\in\mathbb R^2; u_i(x,y)=\lambda_i\}$ such that 
 \begin{equation}\label{gammas}
 \Gamma_i \subseteq \{(x,y)\in\mathbb R^2; \partial_y u_i(x,y)\neq0\}. 
 \end{equation}
 Then, there exist constants  $a_i,b_i\in\mathbb R$ such that 
  \begin{equation}\label{yab}
 \Gamma_i \subseteq \{(x,y)\in\mathbb R^2 ; y=a_ix^2+b_i\}. 
  \end{equation}
  
Consider $(x_*,y_*)\in\Gamma_i$ such that $\partial_y u_i(x_*,y_*)\neq 0$. There exists a smooth function $\gamma_i:(x_*-\epsilon, x_*+\epsilon)\to\mathbb R$ such that $\gamma_i(x_*)=y_*$ in $\Gamma^*_i\subset \Gamma_i$ and 
\begin{equation}
\Gamma^{*}_i \cap \Gamma_i =\{(t,\gamma_i(t) ); x_*-\epsilon<t< x_*+\epsilon \}. 
\end{equation}
We now define 
\begin{equation}
v_i(x,y) := y- \gamma_i(x). 
\end{equation} 
Note that $\nabla_{\mathbb G} v_i(x,y)=(-\gamma'_i(x),x)$ and
\begin{equation}
\partial_xu_i(t,\gamma_i(t)) + \partial_y u_i(t,\gamma_i(t)) \gamma_i'(t) = \partial_t \left[ u_i(t,\gamma_i (t))\right] = \partial_t \left[\lambda_i \right]= 0. 
\end{equation}
From this, we conclude that if $|\nabla_{\mathbb G} v_i(x,y)|=0$ then $x=0$ and $\gamma_i'(0)=0$ that implies $\partial_x u_i(x,y)=0$ and $\partial_y u_i(x,y)=0$ that is $|\nabla_{\mathbb G} u_i(x,y)| = 0$.     Since the intrinsic curvature of the level sets of $u_i$  vanishes identically and the fact that $|\nabla_{\mathbb G} u_i|\neq 0$ and $|\nabla_{\mathbb G} v_i|\neq 0$ in $\Gamma^{*}_i \cap \Gamma_i$, we can use $v_i$ to compute the intrinsic curvature as 
\begin{equation}
0=\mathcal A_i=\div_{\mathbb G} \left[ \frac{(-\gamma'_i(x),x)}{\sqrt{(\gamma '_i(x))^2 + x^2} } \right]
=  - \partial_x \left[ \frac{\gamma'_i(x) }{ \sqrt{ \left(\gamma '_i(x) \right)^2 + x^2} }\right]. 
\end{equation}
Therefore, there exists a constant $C^*_i$ such that for all $x_*-\epsilon<x< x_*+\epsilon$,  
\begin{equation}
 \frac{\gamma'_i(x) }{ \sqrt{ \left(\gamma '_i(x) \right)^2 + x^2} } =C^*_i. 
\end{equation}
This implies that $\gamma'_i(x)$ has a fixed sign and more importantly there exist constants $A^*_i$ and $B_i^*$ such that 
\begin{equation}
\gamma_i(x) =A^*_i x^2 + B^*_i. 
\end{equation}
Applying some continuity arguments, the above formulation of $\gamma_i$ can be extended to the entire $\Gamma_i$.  This completes the proof of  (\ref{yab}).

\noindent {\bf Step 2.}  Suppose that $\Gamma_i$ is a connected component of the level set $\{(x,y)\in\mathbb R^2; u_i(x,y)=\lambda_i\}$ such that (\ref{gammas}) holds. Then, $\Gamma_i$ and $ \{(0,y); y\in\mathbb R\} $ are disjoint sets. 

From Step 1, we know that (\ref{yab}) holds. Suppose that $(0,b_i)\in \Gamma_i$. Then,  $u_i(x,a_i x^2 +b_i)=\lambda_i$ for $|x|<\epsilon$ for some positive $\epsilon$ that implies
 \begin{equation}
 0=\partial_x \lambda_i=\partial_x u_i(x,a_i x^2 +b_i) + 2 a_i x\partial_y u_i(x,a_i x^2 +b_i). 
 \end{equation}
We now set $x=0$ to conclude that $\partial_x u_i(0,b_i) =0$ and therefore 
\begin{equation}
|\nabla_{\mathbb G}(0,b_i)|^2=|\partial_x u_i(0,b_i)|^2+ |(0)\partial_y u_i(0,b_i)|^2 =0,
\end{equation}
that is a contradiction. Therefore,  $\Gamma_i$ and $ \{(0,y); y\in\mathbb R\} $ are disjoint sets.

 \noindent {\bf Step 3.}  Suppose that $\Gamma_i$ is a connected component of the level set $\{(x,y)\in\mathbb R^2; u_i(x,y)=\lambda_i\}$ such that  $\partial_y u_i(x_*,y_*)\neq 0$ for some  $(x_*,y_*)\in\Gamma_i$. Then, 
  \begin{equation}\label{gammayu}
 \Gamma_i \subseteq \{(x,y)\in\mathbb R^2; \partial_y u_i(x,y)\neq0\}. 
 \end{equation}
 
If (\ref{gammayu}) is not true, then there is a point $(\bar x,\bar y)\in\mathbb R^2$ such that $\partial_y u_i(\bar x,\bar y)=0$.  In addition, there is a sufficiently smooth curve $\gamma_i:(\bar x-\epsilon,\bar x+\epsilon)\to\mathbb R$ such that $\gamma_i(\bar x)=\bar y$ that gives 
\begin{equation}\label{yuxy}
\partial_y u_i(\bar x,\bar y)=0. 
\end{equation}
 From Step 1, we conclude that $\gamma_i$ is parabolic that is $\gamma_i(x)=a_i x^2 +b_i$ that implies $u_i(x,a_i x^2 +b_i)=\lambda_i$. Therefore, 
 \begin{equation}
 0=\partial_x \lambda_i=\partial_x u_i(x,a_i x^2 +b_i) + 2 a_i x\partial_y u_i(x,a_i x^2 +b_i). 
 \end{equation}
From this and (\ref{yuxy}),  we obtain  
\begin{equation}\label{xuxy}
\partial_x u_i(\bar x, a_i \bar x^2+b_i)=0. 
\end{equation}
Combining  (\ref{xuxy}) and (\ref{yuxy}), we conclude $|\nabla_{\mathbb G} u_i(\bar x,a_i \bar x^2+b_i)|=0$ that is a contradiction. Therefore, (\ref{gammayu}) holds. 
 
\noindent {\bf Step 4.} For all $(x,y)$ that belongs to the level sets $\{(x,y)\in\mathbb R^2; u_i(x,y)=\lambda_i\}$, we have $\partial_yu_i(x,y)=0$ and $\partial_xu_i(x,y) \neq 0$. 

Suppose that there exists  $(x^*,y^*)$ in the level sets such that $\partial_yu_i(x^*,y^*) \neq 0$. Then, consider  a connected component $\Gamma_i$ of the level sets such that  $\partial_y u_i(x_*,y_*)\neq 0$. From Step 3, we conclude that (\ref{gammayu}) holds. Now, Step 1 yields  $\Gamma_i$ is a parabola with vertex on $(0,b_i)$ that intersects the plane $x=0$ and this contradicts Step 2. This proves $\partial_yu_i(x,y)=0$ for all $(x,y)\in\mathbb R^2$ such that $u_i(x,y)=\lambda_i$, and the fact that $|\nabla_{\mathbb G}u_i(x,y)|\neq 0$ implies $\partial_xu_i(x,y) \neq 0$. 

In order to complete the proof, we consider the following parametrization of the level sets 
\begin{equation}
 u_i(f_i(t),g_i(t))=\lambda_i  \ \ \text{for some} \ \ t\in I\subset\mathbb R.
 \end{equation}
  From this and Step 4, we conclude  
\begin{equation}
\partial_x u_i(f_i(t),g_i(t)) f'_i(t) + \partial_y u_i(f_i(t),g_i(t)) g'_i(t)=  \partial_x u_i(f_i(t),g_i(t)) f'_i(t) =0, 
\end{equation}
that gives $f'_i(t)\equiv 0$ that is $f_i(t)\equiv C_i$ where $C_i$ is a constant for each $1\le i\le m$.  Therefore, the level sets of $u_i$ are vertical straight lines. 
 
 \end{proof}

We now provide a counterpart of Theorem \ref{thmred} when the assumption (\ref{BRNablauR4}) and stability are replaced with the local minimality.   In what follows,  we provide  the notion of minimizers of the energy functional
 \begin{equation}\label{E}
 E(u, \Omega)=\int_{\Omega}  \left[\frac{1}{2} \sum_{i=1}^m  |\nabla_{XY} u_i|^2 + \tilde H(u)\right]dxdy ,
 \end{equation}
where $\tilde H\in C^2(\mathbb R^m)$ such that $\tilde H(\alpha)=\tilde H(\beta)=0\le\tilde H(\cdot)$ for some $\alpha=(\alpha_i)_{i=1}^m$ and $\beta=(\beta_i)_{i=1}^m$.  
\begin{dfn} We say $u=(u)_{i=1}^m\in C^2(\Omega,\mathbb R^m)$ for an open subset $\Omega$ of $\mathbb R^n$ is a local minimum of  (\ref{E})  whenever  
 \begin{equation}
 E(u,\Pi) \le  E(v, \Pi) ,
 \end{equation}
 for bounded open subset $\Pi\subset \Omega$ and $v=(v_i)_{i=1}^m$ such that $v=u$ in $\Omega\setminus\Pi$.  
 \end{dfn}

 \begin{thm} \label{}
 Suppose that  $u=(u_i)_{i=1}^m$ is a local minimum for (\ref{E}) in two dimensions for Grushin vector fields satisfying (\ref{GrXY}).  Assume that $\alpha_i\le u_i\le \beta_i$,  $|\nabla_{\mathbb G} u_i|\in L^\infty$ such that $|\nabla_{\mathbb G} u_i |\neq 0$ for each index $i$,  and  
 \begin{equation}\label{ZYuXu2}
 ZYu_i Xu_i-ZXu_iYu_i \le 0. 
 \end{equation}
  Then, each $u_i$ depends only on the variable $x$.  
  \end{thm}
  \begin{proof} 
  We show that the minimality assumption implies that the growth-decay estimate (\ref{BRNablauR4}) holds.  Consider the following test function for $R>1$ and $0<\epsilon<1$, 
   \begin{equation}\label{psixy}
 \psi_i  (x,y):= \psi_i  (||(x,y)||_{\mathbb G})=\left\{
                      \begin{array}{ll}
                       \beta_i, & \hbox{if $||(x,y)||_{\mathbb G}\ge R$,} \\
                       \alpha_i  , & \hbox{if $||(x,y)||_{\mathbb G}\le R-\epsilon$,}
                                                                       \end{array}
                    \right.
                      \end{equation} 
such that $ | \psi'_i  | \le \epsilon^{-1}$ when $R-\epsilon \le ||(x,y)||_{\mathbb G}\le R$. Note that 
\begin{equation}\label{XNorm}
|X ||(x,y)||_{\mathbb G} | = \frac{|x^3|}{ ||(x,y)||^3_{\mathbb G} }\le 1 \ \ \text{and} \ \ |Y  ||(x,y)||_{\mathbb G} | = \frac{|xy|}{ 2||(x,y)||^3_{\mathbb G} } \le 1. 
\end{equation}
From (\ref{XNorm}) and (\ref{psixy}) in $R-\epsilon \le ||(x,y)||_{\mathbb G}\le R$,  we conclude 
\begin{equation}
|\  ||\nabla _{\mathbb G } \psi_i ||_{\mathbb G} | \le  | \psi'_i  (||(x,y)||_{\mathbb G})  |  \left[|X ||(x,y)||_{\mathbb G} | ^2 + |Y ||(x,y)||_{\mathbb G} | ^2 \right]
\le C \epsilon^{-1},
\end{equation}
where $C$ is independent from $R$ and $\epsilon$. We now define
 \begin{equation}
 v_i(x,y):=\min\{\psi_i  (x,y), u_i(x,y)\}.
 \end{equation}
  Note that $v_i (x,y)=u_i (x,y) $ when  $||(x,y)||_{\mathbb G}\ge R$ and $v_i (x,y)=\alpha_i  $ when  $||(x,y)||_{\mathbb G}\le R-\epsilon$. From the minimality of $u$ we conclude 
\begin{eqnarray}\label{ER}
 E_R(u) &=& \int_{B_R} \left[\frac{1}{2} \sum_{i=1}^m  |\nabla_{\mathbb G} u_i|^2 + \tilde H(u) \right]dxdy
 \\ &\le& \nonumber \int_{B_R}\left[ \frac{1}{2} \sum_{i=1}^m  |\nabla_{\mathbb G} v_i|^2 + \tilde H(v) \right] dxdy
 \\ &\le& \nonumber \int_{B_R\setminus B_{R-\epsilon}} 
 \left[ \frac{1}{2} \sum_{i=1}^m  |\nabla_{\mathbb G} \psi_i|^2 +  \frac{1}{2} \sum_{i=1}^m  |\nabla_{\mathbb G} u_i|^2+  ||\tilde H||_{L^\infty} \right] dxdy 
  \\ &\le& \nonumber C_{\epsilon,\alpha,\beta}  \int_{B_R\setminus B_{R-\epsilon}}  dx dy  , 
\end{eqnarray}
where $C_{\epsilon,\alpha,\beta} $ is independent from $R$ and we have used the assumption that $|\nabla_{\mathbb G} u_i|$ is globally bounded. The above implies that for each index $i$, we have
\begin{equation}\label{ER2}
  \int_{B_R}  |\nabla_{\mathbb G} u_i|^2 dxdy \le C_{\epsilon,\alpha,\beta}  R^{2},
\end{equation} 
where we have used the fact that $B_R=\{(x,y); x^4+y^2<R^4\}$  and $|B_R|=|B_1| R^3$. The above estimate (\ref{ER2}) implies that 
\begin{equation}\label{ER4}
  \int_{B_R} x^2 |\nabla_{\mathbb G} u_i|^2 dxdy \le  R^2 \int_{B_R}  |\nabla_{\mathbb G} u_i|^2 dxdy \le C_{\epsilon,\alpha,\beta}  R^{4}. 
\end{equation} 
Note also that the minimality implies the stability inequality (\ref{stability}). The rest of the proof follows from the proof of Theorem \ref{thmred}. 
  
  \end{proof}

We end this section with briefly  justifying  the following assumption in the previous theorems
\begin{equation}\label{ZXYINE}
2\left[ ZYu_i Xu_i-ZXu_iYu_i  \right]\le 0.
 \end{equation}
Note that in the case of commutative vector fields, including the classical Laplacian operator,  the above is identically zero. Suppose now that $|\nabla_{XY} u_i| \neq 0$.  Without loss of generality, we assume that $|Xu_i| \neq 0$. Therefore, 
 \begin{equation}\label{ZYXu}
 ZYu_i Xu_i-ZXu_iYu_i = |Xu_i|^2 Z\left(\frac{Yu_i}{Xu_i} \right). 
 \end{equation}
Now, consider the function $\phi(y)=x(y)$ and  $\psi(y)= \phi^2(y)$ on  the level sets of each $u_i$ that is $\{u_i(x,y)=\lambda_i\}$.  Therefore, we have $u_i(\phi(y),y)=\lambda_i$ that implies \begin{equation}
 \partial_x u_i \left[ \phi(y) \right]' + \partial_y u_i  =0. 
\end{equation} 
This is equivalent to   
\begin{equation}\label{xuphi}
\frac{1}{2}\partial_x u_i \left[ \psi(y) \right]'+ x(y)\partial_y u_i =0. 
\end{equation}
For the Grushin vector fields, the condition (\ref{ZXYINE})  is equivalent to 
\begin{equation}\label{gruxyle0}
 \partial_y \left(  -x\frac{\partial_x u_i}{\partial_y u_i} \right) \le 0. 
\end{equation}
On the other hand, $ -\frac{\partial_x u_i}{\partial_y u_i}$ is the slope of the level
curves written as a graph over the $x$-axis. Consequently,   the condition (\ref{gruxyle0}) can be regarded as  the slopes of the upper points on the curve (in
$y$)  are smaller for $x > 0$ and larger for $x < 0$ if one writes the level curves of $u_i$ as graphs over the $x$-axis.

\end{document}